\documentclass [11pt,a4paper]{article}
\usepackage[ansinew]{inputenc}
\usepackage{amssymb}
\usepackage{amsmath}
\usepackage{graphicx}
\usepackage{amsthm}
\usepackage{cite}
\usepackage{color}

\newtheorem{theorem}{\textbf{Theorem}}[section]
\newtheorem{lemma}{\textbf{Lemma}}[section]
\newtheorem{proposition}{\textbf{Proposition}}[section]
\newtheorem{corollary}{\textbf{Corollary}}[section]
\newtheorem{remark}{\textbf{Remark}}[section]
\newtheorem{definition}{\textbf{Definition}}[section]

\allowdisplaybreaks[4]

\def\be{\begin{equation}}
\def\ee{\end{equation}}
\def\bea{\begin{eqnarray}}
\def\eea{\end{eqnarray}}
\def\bt{\begin{theorem}}
\def\et{\end{theorem}}
\def\bl{\begin{lemma}}
\def\el{\end{lemma}}
\def\br{\begin{remark}}
\def\er{\end{remark}}
\def\bp{\begin{proposition}}
\def\ep{\end{proposition}}
\def\bc{\begin{corollary}}
\def\ec{\end{corollary}}
\def\bd{\begin{definition}}
\def\ed{\end{definition}}

\def\non{\nonumber }
\def\tth{\tilde{\theta}}
\def\tch{\tilde{\chi}}

\def\tq{\mathbf{q}}
\def\ve{\varepsilon}

\def\txi{\tilde{\xi}}

\begin{document}

\title{Non-isothermal viscous Cahn--Hilliard equation with inertial term and dynamic boundary conditions}

\author{
{\sc Cecilia Cavaterra}\footnote{Dipartimento di Matematica, Universit\`{a} degli
Studi di Milano, Milano 20133, Italy, {\it cecilia.cavaterra@unimi.it}},
\ {\sc Maurizio Grasselli}\footnote{Dipartimento di Matematica, Politecnico di Milano, Milano 20133, Italy,
{\it maurizio.grasselli@polimi.it}}, \
{\sc Hao Wu}\footnote{School of Mathematical Sciences and Shanghai Key Laboratory for Contemporary Applied
Mathematics, Fudan University, Shanghai 200433, P.R. China,  \textit{haowufd@yahoo.com}}
}

\date{\today}

\maketitle


\begin{abstract}\noindent
We consider a non-isothermal modified Cahn--Hilliard equation which was previously analyzed by M. Grasselli et al.
Such an equation is characterized by an inertial term and a viscous term and it is coupled with a hyperbolic heat equation.
The resulting system was studied in the case of no-flux boundary conditions.
Here we analyze the case in which the order parameter is subject to a dynamic boundary condition.
This assumption requires a more refined strategy to extend the previous results to the present case.
More precisely, we first prove the well-posedness for solutions with bounded energy as well as for weak solutions.
Then we establish the existence of a global attractor. Finally, we prove the convergence of any given weak solution to a single equilibrium
by using a suitable \L ojasiewicz--Simon inequality.


\medskip\noindent
\textbf{MRS 2010:} 35B40, 35B41, 37L99, 80A22.
\end{abstract}

\section{Introduction}

 \setcounter{equation}{0}

The Cahn--Hilliard equation is a cornerstone in Materials Science since it gives a fairly good description of phase
separation processes in binary alloys (see, e.g., \cite{CMZ,NC2,NC3} and references therein).
The early stage of such a phenomenon is called spinodal decomposition.
A modification of the Cahn--Hilliard equation has been proposed in \cite{Galen} to account for rapid spinodal decomposition in
certain materials (see also \cite{GJ1,GJ2}). This modified equation reads as follows
$$
\ve \chi_{tt}+\chi_t-\Delta\mu =0,
$$
where $\ve>0$ is a relaxation time, $\chi$ represents the (relative) concentration of one component and $\mu$ is the
so-called chemical potential given by
$$
\mu=-\Delta\chi+\alpha \chi_t+f(\chi).
$$
Here $\alpha > 0$ is a viscosity parameter accounting for possible presence of microforces (see \cite{NC1}) and $f$ is the
derivative of a given double-well potential. We recall that the classical Cahn--Hilliard equation corresponds to the case
$\ve=\alpha=0$. The case $\epsilon>0$ and $\alpha=0$ is a very challenging equation (see \cite{GSZ09,GSSZ09,GSZ10,Se}, cf. also
\cite{BGM2,GGMP051D, ZM1,ZM2} for the 1D case) which becomes much nicer in presence of viscosity (cf. \cite{Bo,BGM1,BGM3,GGMP05,K})
In particular, in the latter case, solutions regularize in finite time. Moreover, when $(\ve,\alpha)$ tends to zero and $\alpha$
dominates $\epsilon$, then the modified viscous Cahn--Hilliard equation (MVCH) is very close to the standard one in a rigorous way (see \cite{BGM1,BGM3,GGMP05}).
A non-isothermal version of MVCH equation has been proposed and analyzed in \cite{GPS2} (cf. also \cite[9.1.5]{BS}), namely,
 \bea
 &&(\theta+\chi)_t+\nabla \cdot \tq=0,\qquad \text{in \, }\Omega\times (0,\infty),\label{1}\\
 && \sigma\tq_t+\tq=-\nabla \theta,\qquad\qquad \text{in \, }\Omega\times (0,\infty),\label{2}\\
 &&\ve \chi_{tt}+\chi_t-\Delta\mu=0,\ \qquad \text{in \, }\Omega\times (0,\infty),\label{3}\\
 && \mu=-\Delta\chi+\alpha \chi_t+f(\chi)-\theta, \qquad \text{in \, }\Omega\times (0,\infty),\label{4}
 \eea
where $\theta$ represents the (relative) temperature, $\tq$ is the heat flux which is given by the Maxwell--Cattaneo's law \eqref{3}, $\sigma>0$ is a further relaxation time and $\Omega \subset \mathbb{R}^d$ ($d=2,3$) is a bounded domain with a smooth boundary $\Gamma$.

System \eqref{1}--\eqref{4} has been endowed in \cite{GPS2} with no-flux boundary conditions. Here we want to
consider the same system subject to the following boundary conditions
 \bea
 && \tq\cdot \nu =\partial_\nu\mu=0,\qquad \text{on \, }\Gamma\times (0,\infty),\label{5}\\
 &&  \chi_t-\Delta_\Gamma \chi+\partial_\nu \chi+g(\chi)=0, \qquad \text{on \, } \Gamma\times (0,\infty),\label{6}
 \eea
where $\nu$ stands for the outward normal unit vector on the boundary and $\Delta_\Gamma$ stands for the Laplace--Beltrami
operator. The system is also subject to the initial conditions
 \be
 \theta(0)=\theta_0, \quad \tq(0)=\tq_0,\quad
 \chi(0)=\chi_0,\quad \chi_t(0)=\chi_1,\quad \text{in \, }  \Omega.\label{7}
 \ee
We recall that dynamic boundary conditions like \eqref{6} have been proposed by physicists to take into account possible
interactions between the binary alloy and the container walls (see, e.g., \cite{FMD1,FMD2,K-etal}). From the mathematical viewpoint, the
Cahn--Hilliard equation with dynamic boundary conditions has been analyzed in a number of papers (cf., e.g.,
\cite{CFJ06,GaW08,GMS09,GMS,MZ05,MZ10,RZ03,WH07,WZ04}, see also \cite{Ga07,Ga08,GM1,GM2} for the non-isothermal case).
However, the MVCH equation with dynamic boundary conditions
has only been considered in the isothermal case. In \cite{CGG} the authors studied a slightly more general equation
with memory which reduces to the MVCH equation if the kernel is a decreasing exponential. They proved well-posedness, regularity,
and the existence of global and exponential attractors. More recently, the construction of a family of exponential attractors which is robust
with respect to the relaxation time (say $\ve$) has been established in \cite{GG13}.

Here we want to extend the results proven in \cite{GPS2}, namely, well-posedness, existence of the global attractor and convergence
to a single equilibrium. More precisely, we first establish the existence and the uniqueness of global (bounded) energy and weak solutions.
We recall that bounded energy solutions are more general than weak solutions (cf. Definition \ref{energy-weak} below). In addition, in the present case a regularizing effect for $\chi$ is missing due to the presence of the dynamic boundary condition \eqref{6}. This entails that the equation \eqref{3}
must be understood in a more generalized way with respect to \cite{GPS2} (see Remark \ref{regularity} below). In this case the application of the
\L ojasiewicz--Simon technique is also more complicated than in \cite{GPS2} and it seems necessary to work with weak solutions (cf. \eqref{lsa} and
\eqref{LSA} below).

The plan of the paper goes as follows. In the next section the main assumptions as well as the notions of energy and weak solutions are introduced.
In Section \ref{sec3} some a priori energy and higher-order uniform estimates are obtained.
Then, existence and uniqueness of energy and weak solutions are proven. Section \ref{sec4} is devoted to establish the existence of the
global attractor for the semigroup acting on the energy phase space. Finally, in Section \ref{sec5} the convergence of a weak solution to a single
equilibrium is analyzed. Among the open issues it is worth mentioning the existence of a family of exponential attractors and its robustness with respect to $\sigma,$ $\ve$ and $\alpha$ (see \cite{GGMP05} for the isothermal case).

\section{Preliminaries}
\label{sec2}
\setcounter{equation}{0}

Due to the presence of the Laplace-Beltrami operator $\Delta_\Gamma$, in order to deal with system \eqref{1}-\eqref{7},
it is convenient to introduce the unknown function $\xi:=\chi|_\Gamma$ defined on the
boundary $\Gamma$. Setting $\xi_0:=\chi_0|_\Gamma$, we can rewrite the original system as
\bea
 &&(\theta+\chi)_t+\nabla \cdot \tq=0,\qquad \text{in \, } \Omega \times (0,\infty), \label{c1}\\
 && \tq_t+\tq=-\nabla \theta,\qquad \text{in \, } \Omega \times (0,\infty), \label{c2}\\
 &&  \chi_{tt}+\chi_t-\Delta\mu=0, \qquad \text{in \, } \Omega \times (0,\infty), \label{c3}\\
 && \mu=-\Delta\chi+\alpha \chi_t+f(\chi)-\theta, \qquad \text{in \, } \Omega \times (0,\infty), \label{c4}\\
 && \tq\cdot \nu =\partial_\nu\mu=0,\qquad \text{on \, } \Gamma\times (0,\infty),\label{c5}\\
 &&  \xi_t-\Delta_\Gamma \xi +\partial_\nu \chi+g(\xi)=0, \quad \text{on \, } \Gamma\times (0,\infty),\label{c6}\\
 &&\theta(0)=\theta_0, \quad \tq(0)=\tq_0,\quad
 \chi(0)=\chi_0, \quad \xi(0)=\xi_0,\quad \chi_t(0)=\chi_1,\quad \text{in \, } \Omega.\hskip2truecm \label{c7}
 \eea
For the sake of simplicity, here and in the remaining part of the paper we assume $\ve =\sigma= 1$.
Besides, we will consider only the viscous case $\alpha >0$ even though existence of an energy solution can be
proven also in the case $\alpha=0$.

\textbf{Notations and functional spaces.}
We denote by $|\Omega|$ the Lebesgue measure of $\Omega$ and by $|\Gamma|$ the $n-1$-dimensional measure of $\Gamma$.
For a given real Banach space $X$, its norm is indicated by $\|\cdot\|_{X}$.
The symbol $(\cdot, \cdot)_{X,X^*}$ stands for a duality pairing between the Banach space $X$ and its dual $X^*$.
We denote by $L^p(\Omega)$ and $L^p(\Gamma)$ $(p\geq 1)$ the standard Lebesgue spaces with respective norms $\|\cdot\|_{L^p(\Omega)}$
and $\|\cdot\|_{L^p(\Gamma)}$.
For $s>0$, $H^s(\Omega)$ and $H^s(\Gamma)$  stand for
the Sobolev spaces normed by $\|\cdot\|_{H^s(\Omega)}$ and $\|\cdot\|_{H^s(\Gamma)}$.
Bold letters are used to denote the corresponding vector spaces, for instance,
$\mathbf{L}^2(\Omega)=(L^2(\Omega))^d$, $\mathbf{H}^1(\Omega)=(H^1(\Omega))^d$.

For the sake of brevity, the norm in $L^2(\Omega)$ and $\mathbf{L}^2(\Omega)$ will be simply indicated by $\|\cdot\|$
and the inner products in $L^2(\Omega)$ and $L^2(\Gamma)$ will be denoted by
$(\cdot, \cdot)$ and $(\cdot, \cdot)_{L^2(\Gamma)}$, respectively.

Besides, we set
\bea
&& H = L^2(\Omega),\quad  H_\Gamma=L^2(\Gamma),\quad V = H^1(\Omega), \quad V_\Gamma=H^1(\Gamma),\non\\
&& \mathbf{V}_0=\{\mathbf{v}\in \textbf{H}^1(\Omega) \ : \ \mathbf{v}\cdot\nu|_\Gamma=0\},\non\\
&& H_0=\left\{v\in H: \langle v\rangle := |\Omega|^{-1}\int_\Omega v dx=0\right\}, \non
\eea
and we introduce the Hilbert space $\mathbf{L}^2_{{\rm div}}(\Omega)$ and its inner product
$$ \mathbf{L}^2_{{\rm div}}(\Omega) =\{\mathbf{q}\in \mathbf{L}^2(\Omega): \nabla \cdot \mathbf{q}\in L^2(\Omega)\},
\, \, \, (\tq_1, \tq_1)_{\mathbf{L}^2_{{\rm div}}(\Omega)}=(\tq_1, \tq_2)_{\mathbf{L}^2(\Omega)}+(\nabla \cdot \tq_1, \nabla \cdot \tq_2)_{L^2(\Omega)}.
$$
It is well known that if $\tq\in \mathbf{L}^2_{{\rm div}}(\Omega)$ then $\tq\cdot \nu\in H^{-\frac12}(\Gamma)$ (cf. \cite{Monk}).
Hence we introduce the following closed subspace of $\mathbf{L}^2_{{\rm div}}(\Omega)$
$$
 \mathbf{W}_0=\{\tq\in \mathbf{L}^2_{{\rm div}}(\Omega): \tq\cdot \nu|_\Gamma=0\}.
$$
We have ${\mathbf W}_0\hookrightarrow \mathbf{L}^2(\Omega)\hookrightarrow ({\mathbf W}_0)^*$ with dense and continuous
embeddings.

The Laplace operator with Neumann boundary condition and its domain are denoted by
$$A=-\Delta :D(A)\subset H \to H_0, \quad D(A) = \{v \in H^2(\Omega):\partial _\nu v=0 \ \text{on}\ \Gamma\},$$
and we indicate with $A_0$ its restriction to $H_0$. Note that $A_0$ is a positive linear operator.  Hence, for any
$r\in\mathbb{ R}$, we can define its powers $A^r_0$ and their domains $D(A^ \frac{r}{2}_ 0 )$, setting
$$V^ r_ 0 = D(A^ \frac{r}{2}_ 0 ), \ \text{with inner product} \ \ (v_1,v_2)_{V_0^r}=( A_0^\frac{r}{2}v_1,A_0^\frac{r}{2}v_2).$$
Taking any $u\in V^*$ with  $\langle u\rangle=0$, then $v=A_0^{-1} u$ is a solution to the
generalized Neumann problem for $A$ with source $u$ and the restriction $\langle v\rangle=0$.
Hence, for any $u, w\in V^*$ with $\langle u\rangle=\langle w\rangle=0$, we have
$$ (u, A_0^{-1} w)_{V^*,V}= (w, A_0^{-1} u)_{V^*, V}=\int_\Omega (\nabla  A_0^{-1} u)\cdot (\nabla  A_0^{-1} w) dx. $$
We endow $V^*$ with the equivalent norm
$\|v\|^2_{V^*}=\|\nabla A_0^{-1}(v-\langle v\rangle)\|^2+ |\langle v\rangle|^2$, for any $v\in V^*$.
Moreover, if $u\in H^1(0,T; V^*)$ with $\langle u\rangle=0$, then
$$ (u_t, A^{-1}_0 u)_{V^*,V}=\frac12\frac{d}{dt}\|u\|_{V^*}^2, \quad \text{a.e.}\ t\in (0, T).$$

Next, we introduce the product spaces
$$\mathbb{H}=H \times H_\Gamma, \quad H^r(\Omega)\times H^r(\Gamma)$$
and the subspaces of $H^r(\Omega)\times H^r(\Gamma)$
$$ \mathbb{H}^r:=\{ (\chi, \xi) \in H^r(\Omega)\times H^r(\Gamma) \ : \ \xi = \chi|_\Gamma \}, \quad \forall \, r>\frac12,$$
with the induced graph norm. We note that $h=(u,v) \in \mathbb{H}$ will be thought as a pair of functions
belonging, respectively, to $H$ and to $H_\Gamma$. If we do not have additional regularity, the second component of $h$ (i.e., $v$) is not necessary to be
the trace of the first one (i.e. $u$). The elements of $\mathbb{H}^r$ will be considered as pairs of functions $(\chi, \chi|_\Gamma)$ such that $\mathbb{H}^r$
is identified with a (closed) subspace of the product space $H^r(\Omega)\times H^r(\Gamma)$.
For $r_1>r_2>\frac12$, the dense and compact embeddings $\mathbb{H}^{r_1}\hookrightarrow \mathbb{H}^{r_2}$ hold.
Finally, we introduce the closed subspaces of $\mathbb{H}$ and $\mathbb{H}^r$ as follows
$$ \mathbb{H}_0=\{(u, v)\in \mathbb{H}: \ \langle u \rangle=0\}, \quad \mathbb{H}^r_0=\{(u, v)\in \mathbb{H}^r: \ \langle u \rangle=0\},\quad\forall \, r >\frac12.$$
According to the structure of system \eqref{c1}--\eqref{c7}, we define the product spaces
$$
\mathbb{X}= H\times \mathbf{H} \times \mathbb{H}^1\times V^*,\qquad
\mathbb{Y}= V\times \mathbf{V}_0\times \mathbb{H}^3 \times V,\non
$$
endowed with the following norms
\bea
\|(z_1,\mathbf{z}_2,z_3, z_4, z_5)\|^2_{\mathbb{X}}
&=&\|z_1\|^2+
\|\mathbf{z}_2\|^2+\|(z_3, z_4)\|_{\mathbb{H}^1}^2+\|z_5\|_{V^*}^2,\\
\|(z_1,\mathbf{z}_2,z_3,z_4,z_5)\|^2_{\mathbb{Y}}&=&
\|z_1\|_V^2+
\|\mathbf{z}_2\|_{\mathbf{H}^1(\Omega)}^2+\|(z_3, z_4)\|_{\mathbb{H}^3}^2+\|z_5\|_V^2.\non
\eea
It is easy to see that the continuous embedding $\mathbb{Y}\hookrightarrow\mathbb{X}$ holds.

\textbf{Assumptions on the nonlinearities.}
Let us now list our assumptions on $f$ and $g$.
 \begin{itemize}
 \item[(H1)] $ f, \, g\in C^2(\mathbb{R})$,
 \item[(H2)] Dissipative condition: $\displaystyle\liminf_{|s|\to+\infty} f'(s)>0, \; \displaystyle\liminf_{|s|\to+\infty} g'(s)>0,$
 \item[(H3)] Growth condition:
 $$|f''(y)|\leq c_f(1+|y|^p),\quad |g''(y)|\leq c_g(1+|y|^q),\quad  \forall \, y\in \mathbb{R},$$
 for some generic positive constants $c_f, c_g$ independent of $y$, with $q\in [0, +\infty)$ and $p\in[0,1]$ when $n=3$,
 while $p\in [0, +\infty)$ for $n=2$.
\end{itemize}
\br \label{AH}
Consider the potential functions $F(y)=\int_0^y f(s)ds$ and $G(y)=\int_0^y g(s)ds$,  $y\in \mathbb{R}$.
It is easy to check that assumptions (H2)--(H3) yield the following properties (cf. e.g., \cite{GM1}):
 \\
(1) there exist $c_0,c_1>0$ such that
$$ f'(y)\geq -c_0, \quad F(y)\geq -c_1, \quad \forall \, y\in \mathbb{R},$$
(2) for any $M_0 \in \mathbb{R}$, there exist $c_2, c_3>0$ and
a sufficiently large $c_4>0$ such that $$ (y-M_0)f(y) \geq  c_2 (y-M_0)^2+c_3F(y)-c_4,\quad
\forall\, y\in \mathbb{R},$$
(3) $\forall\, \epsilon>0$, there exists $c_\epsilon>0$ sufficiently large such that $$|f(y)|\leq
 \epsilon F(y)+c_\epsilon, \ \forall\, y\in \mathbb{R}.$$
\er
Similar results hold also for the potential $G(y)$.
\br
One can verify, for instance, that the classical double well potential $F(y)=\frac14(y^2-1)^2$ and the corresponding function $f(y)=y^3-y$
satisfy (H1)--(H3) while $g$ can be any polynomial of odd degree with a positive leading coefficient.
 \er

We are ready to introduce the variational formulation of problem \eqref{c1}--\eqref{c7}.

\begin{definition}
\label{energy-weak}
Let $T\in (0, +\infty)$. The set of functions
 $(\theta, \mathbf{q}, \chi, \xi, \chi_t)$ satisfying
\bea
 && (\theta, \mathbf{q}, \chi, \xi, \chi_t)\in L^\infty(0,T; \mathbb{X}),  \label{r1}\\
 && \theta_t\in L^\infty(0, T; V^*),\quad  \tq_t\in L^2(0,T; (\mathbf{V}_0)^*),\label{r2}\\
 &&  \chi_t\in L^2(0, T; V^*), \quad \alpha^\frac12 \chi_t\in L^2 (0,T; H),
 \quad \xi_t\in L^2(0, T; H_\Gamma),\label{r3} \\
 &&  \chi_{tt}+\chi_t\in L^\infty(0, T; D(A_0^{-\frac32})),\label{r4}
\eea
is an \emph{energy solution} to problem \eqref{c1}--\eqref{c5} with initial datum $(\theta_0, \mathbf{q}_0, \chi_0, \xi_0, \chi_1)\in \mathbb{X}$,
if the following identities hold, for a.e. $t\in (0, T)$,
 \bea
 && ((\theta+\chi)_t, w)_{V^*, V}-(\mathbf{q}, \nabla w)=0, \label{ee1}\\
 && (\mathbf{q}_t+\mathbf{q}, \mathbf{v})_{\mathbf{V}_0^*, \mathbf{V}_0}-(\theta, \nabla \cdot \mathbf{v})=0,\label{ee2}\\
 && (A_0^{-1}(\chi_{tt}+\chi_t), \phi)_{V^*,V}+(\mu, \phi)_{V^*,V}=0,\label{we3} \\
 && (\mu_, \phi)_{V^*,V}=(\nabla \chi, \nabla \phi)+(\nabla_\Gamma \xi, \nabla_\Gamma v)_{\mathbf{L}^2(\Gamma)}  +\alpha (\chi_t, \phi)+(\xi_t, v)_{L^2(\Gamma)} \non\\
 && \quad \quad \quad \quad +(f(\chi),\phi)+(g(\xi), v)_{L^2(\Gamma)} -(\theta, \phi),  \label{ew4}
 \eea
for any $w\in V$, $\mathbf{v}\in \mathbf{V}_0$ and $(\phi,v)\in \mathbb{H}^1$ with $v=\phi|_{\Gamma}$.

If, in addition, $(\theta_0, \mathbf{q}_0, \chi_0, \xi_0, \chi_1)\in \mathbb{Y}$ and
\bea
 && (\theta, \mathbf{q}, \chi, \xi, \chi_t)\in L^\infty(0,T; \mathbb{Y}),  \label{r1a}\\
 && \theta_t\in L^\infty(0, T; H),\quad  \tq_t\in L^2(0,T; \mathbf{L}^2(\Omega)),\label{r2a}\\
 &&  \chi_{tt}\in L^2(0, T; V^*), \quad \alpha^\frac12 \chi_{tt}\in L^2 (0,T; H),
 \quad \xi_{tt}\in L^2(0, T; H_\Gamma),\label{r3a} \\
 &&  \chi_{tt}+\chi_t\in L^\infty(0, T; D(A_0^{-\frac12})),\label{r4a}
\eea
then $(\theta, \mathbf{q}, \chi, \xi, \chi_t)$ is a \emph{weak solution} to problem \eqref{c1}--\eqref{c5}.
\end{definition}
\begin{remark}
 We note that, due to the regularities \eqref{r1}--\eqref{r4}, an energy solution belongs to the class
 $ C_w([0,T]; \mathbb{X})$, where the space $C_w([0,T]; X)$ ($X$ being a real Banach space) is defined as
 $$C_w([0,T]; X):=\{ v\in L^\infty(0,T; X): (\phi, v(\cdot))_{X^*, X}\in C^0([0,T]), \ \forall\, \phi\in X^*\}.$$
 Therefore, any energy solution can be evaluated point-wisely in time and initial conditions have a well-defined meaning.
 The same property holds for weak solutions.
 \end{remark}

 \begin{remark}
 \label{regularity}
  Note that in case of homogeneous Neumann boundary conditions we can recover the additional regularity $\chi \in L^2(0,T;H^2(\Omega))$ for
  an energy solution (see \cite[(2.18)]{GPS2}). Thus equation \eqref{we3} can be written in the standard weak form (see \cite[(2.6)]{GPS2}).
  However, in the present case, it seems that this regularity does not hold. On the other hand, such a property is crucial to prove that solutions
  to the isothermal MVCH regularize in finite time (see \cite{Bo}, cf. also \cite{GPS2} for the non-isothermal case with Fourier heat conduction).
  We also point out that the present notion of weak solution is a quasi-strong solution in the terminology introduced in \cite{GSSZ09}.
 \end{remark}

\section{Well-posedness}\setcounter{equation}{0}
\label{sec3}
\subsection{\emph{A priori} estimates}

\textbf{Conserved quantities.}
Integrating \eqref{c1} and \eqref{c3} over $\Omega$, we deduce from the no-flux boundary condition \eqref{c5} that the following
relations hold
\bea
 && \int_\Omega (\theta(t)+\chi(t))dx=\int_\Omega
(\theta_0+\chi_0)dx,\quad \forall\, t\geq 0,\label{cv1}\\
&& \int_\Omega (\chi_t(t)+\chi(t))dx=\int_\Omega (\chi_1+\chi_0)dx, \quad \forall\, t\geq 0.\label{cv2}
 \eea
 The second relation \eqref{cv2} is an ODE for $\langle
 \chi(t) \rangle$, then we have
 \be
 \langle \chi(t)\rangle=\langle \chi_0\rangle+\langle \chi_1\rangle
 -e^{-t}\langle \chi_1\rangle, \quad \langle \chi_t(t)\rangle= e^{-t}\langle \chi_1\rangle.\label{cv3}
 \ee
 It is easy to see that if $\langle \chi_1\rangle=0$, then the so-called mass conservation relation holds
 $$ \int_\Omega \chi(t)dx=\int_\Omega \chi_0dx.$$
 Based on the above observations, in order to obtain dissipative estimates of the solutions to problem \eqref{c1}--\eqref{c7},
 it is convenient to introduce the new variables
 \be
 \label{cv4}
 \tth=\theta-\langle \theta \rangle, \quad \tch
 =\chi-\langle \chi \rangle,\quad \txi=\tch|_\Gamma=\xi-\langle \chi \rangle,
 \ee
 which imply that
 \bea
 &\tch_t(t) =\chi_t(t)-\langle \chi_t (t)\rangle=\chi_t(t)- Q_1(t),\non\\
 &\tch_{tt}(t) = \chi_{tt}(t)-\langle \chi_{tt} (t) \rangle=\chi_{tt}(t)+ Q_1(t),\non
 \eea
 with the function $Q_1$ given by
 \be
 Q_1(t)= \langle \chi_1\rangle e^{-t}.\non
 \ee
Then, system \eqref{c1}--\eqref{c7} can be rewritten as
 \begin{align}
 &(\tth+\tch)_t+\nabla \cdot \tq=0, \quad \text{in \, } \Omega\times (0,\infty),\label{c1a}\\
 & \tq_t+\tq=-\nabla \tth, \quad \qquad \ \text{in \, } \Omega\times (0,\infty), \label{c2a}\\
 & \tch_{tt}+\tch_t-\Delta\tilde{\mu} =0, \quad\ \ \text{in \, } \Omega\times (0,\infty),\label{c3a}\\
 & \tilde{\mu}=-\Delta\tch+\alpha \tch_t+f(\chi)-\tth, \qquad \qquad \text{in \, } \Omega\times (0,\infty), \label{c4a}\\
 &  \txi_t-\Delta_\Gamma \txi+g(\xi)+\partial_\nu \tch+Q_1(t)=0, \quad \text{on \, } \Gamma\times (0,\infty), \label{c6a}\\
 & \tq\cdot \nu =\partial_\nu\tilde{\mu}=0, \quad \text{on \, } \Gamma\times (0,\infty),\label{c5a}\\
 &\tth(0)=\theta_0 - \langle \theta_0 \rangle, \quad \tq(0)=\tq_0,\\
 & \tch(0)=\chi_0 - \langle \chi_0 \rangle, \quad \text{in}\ \Omega,
\quad  \txi(0)=\xi_0 - \langle \chi_0 \rangle,\quad  \tch_t(0)=\chi_1-\langle \chi_1 \rangle,\quad \text{in}\ \Omega.\hskip2truecm \label{c7a}
  \end{align}

  \noindent \textbf{Dissipative estimates.} In what follows, we will derive some uniform estimates on the solutions of problem \eqref{c1}--\eqref{c7}
  which are necessary for studying the well-posedness and long-time behavior of the system.
  The following calculations have a formal character but they can be justified by working within a proper Faedo--Galerkin approximation scheme
  (see \cite{GMS09} and \cite{GMS}).

\bl[Dissipative estimate in $\mathbb{X}$] \label{es}
Let the assumptions (H1)--(H3) be satisfied. Suppose $(\theta(t), \mathbf{q}(t), \chi(t), \xi(t), \chi_t(t))$ is a regular solution of
system \eqref{c1}--\eqref{c7}. Then there exists a positive nondecreasing function $\mathcal{Q}$ such that
\bea
&& \|(\theta(t), \mathbf{q}(t), \chi(t), \xi(t), \chi_t(t))\|_\mathbb{X}^2\non\\
&& \ \ +\int_t^{t+1} (\alpha \|\chi_t(\tau)\|^2
+ \|\xi_t(\tau)\|_{H_\Gamma}^2+ \|(\theta(\tau), \mathbf{q}(\tau), \chi(\tau), \xi(\tau), \chi_t(\tau))\|_\mathbb{X}^2)d\tau\non\\
& \leq& \mathcal{Q}(\|(\theta_0, \mathbf{q}_0, \chi_0, \xi_0, \chi_1)\|_\mathbb{X})e^{-\rho_1t}+\rho_2, \quad \forall\, t\geq 0,\label{disa1}
\eea
where the positive constants $\rho_1, \rho_2$ may depend on $\langle\theta_0\rangle $, $\langle\chi_0\rangle$, $\langle\chi_1\rangle$,
$|\Omega|$, $|\Gamma|$, but are independent of $t$.
In particular, the constants $\rho_1, \rho_2$ are independent of $\|(\theta_0, \mathbf{q}_0, \chi_0, \xi_0, \chi_1)\|_\mathbb{X}$.
\el
 \begin{proof}
Multiplying \eqref{c1a} and \eqref{c2a} by $\tth$ and $\tq$, respectively, and integrating over $\Omega$, we obtain
\bea
&&\frac12\frac{d}{dt}\|\tth\|^2-\int_\Omega \tq\cdot \nabla\tth dx =-\int_\Omega \tch_t\tth dx,\label{e1}\\
&&\frac{1}{2}\frac{d}{dt}\|\tq\|^2 +\|\tq\|^2+\int_\Omega \tq \cdot \nabla \tth dx =0.\label{e2}
\eea
Multiplying \eqref{c3a} by $A_0^{-1}\tch_t$ and integrating over $\Omega$, we have
\bea
&&\frac{d}{dt}\left[\frac{1}{2}\|A_0^{-\frac12}\tch_t\|^2 + \frac12\|\nabla \tch\|^2+\int_\Omega F(\chi)dx
+ \frac12\|\nabla_\Gamma \txi\|^2_{H_\Gamma}+ \int_\Gamma G(\xi)dS\right]\non\\
&& +\|A_0^{-\frac12}\tch_t\|^2+\alpha \|\tch_t\|^2+\|\txi_t\|^2_{H_\Gamma}\non\\
&=& \int_\Omega\tth\tch_tdx+Q_1\left(\int_\Omega f(\chi)dx+\int_\Gamma g(\xi)dS\right) - Q_1\int_\Gamma \txi_t dS.\label{e3}
\eea
In a similar manner, multiplying \eqref{c3a} by $ A_0^{-1}\tch$ and integrating over $\Omega$, we get
\bea
&& \frac{d}{dt}\left(\int_\Omega A_0^{-\frac12}\tch_t A_0^{-\frac12}\tch dx
+ \frac12\| A_0^{-\frac12}\tch\|^2+\frac{\alpha}{2}\|\tch\|^2+\frac{1}{2}\|\txi\|^2_{H_\Gamma}\right)\non\\
&& -\| A_0^{-\frac12}\tch_t\|^2+\|\nabla  \tch\|^2+\|\nabla_\Gamma\txi\|^2_{H_\Gamma} +\int_\Omega f(\chi)\tch dx+\int_\Gamma g(\xi)\txi dx \non\\
&=& \int_\Omega \tth\tch dx-Q_1\int_\Gamma \txi dS.\label{e4}
\eea
Besides, using the equation \eqref{c1a} and \eqref{c2a}, we deduce the identity
\bea
&& \frac{d}{dt}\int_\Omega \mathbf{q}\cdot \nabla A_0^{-1}\tth dx+\|\tth\|^2 \non\\
&=&\int_\Omega \mathbf{q}_t \cdot  \nabla A_0^{-1}\tth dx+ \int_\Omega \mathbf{q} \cdot   \nabla A_0^{-1}\tth_t dx+\|\tth\|^2\non\\
&=& -\int_\Omega \mathbf{q}\cdot \nabla A_0^{-1} \tth dx -\int_\Omega \mathbf{q}\cdot \nabla A_0^{-1}\tch_tdx
+\|A_0^{-\frac12}\nabla \cdot \mathbf{q}\|^2.\label{e5}
\eea
Multiplying \eqref{e4} and \eqref{e5} by some small constants $\kappa_1, \kappa_2>0$ (to be chosen later), respectively, and adding the resulting equations with \eqref{e1}--\eqref{e3}, one deduces that
\bea
&& \frac{d}{dt}\left(\frac12\|\tth\|^2+\frac{1}{2}\|\tq\|^2
+\frac{1}{2}\|A_0^{-\frac12}\tch_t\|^2+ \frac12\|\nabla\tch\|^2+\int_\Omega F(\chi)dx \right.\non\\
&&  +\frac12\|\nabla_\Gamma\txi\|_{H_\Gamma}^2+\frac{\kappa_1}{2}\|\txi\|_{H_\Gamma}^2
+ \int_\Gamma G(\xi)dS+\kappa_1\int_\Omega A_0^{-\frac12}\tch_t A_0^{-\frac12}\tch dx\non\\
&&\left. +\frac{\kappa_1}{2}\| A_0^{-\frac12}\tch\|^2+\frac{\kappa_1\alpha}{2}\|\tch\|^2
+ \kappa_2\int_\Omega \mathbf{q}\cdot \nabla A_0^{-1}\tth dx \right)\non\\
&&  +\|\tq\|^2+(1-\kappa_1)\|A_0^{-\frac12}\tch_t\|^2+\alpha \|\tch_t\|^2+\|\txi_t\|^2_{H_\Gamma} +\kappa_1\|\nabla \tch\|^2\non\\
&& +\kappa_1\|\nabla_\Gamma \txi\|^2_{H_\Gamma} +\kappa_1\left(\int_\Omega f(\chi)\tch dx+\int_\Gamma g(\xi)\txi dS\right)+\kappa_2\|\tth\|^2\non\\
& =& Q_1\left(\int_\Omega f(\chi) dx+\int_\Gamma g(\xi) dS\right)-Q_1\int_\Gamma \txi_t dS
+ \kappa_1\int_\Omega \tth\tch dx-\kappa_1 Q_1\int_\Gamma \txi dS\non\\
&& -\kappa_2\int_\Omega \mathbf{q}\cdot \nabla A_0^{-1} \tth dx -\kappa_2\int_\Omega \mathbf{q}\cdot \nabla A_0^{-1}\tch_tdx
+\kappa_2\|A_0^{-\frac12}\nabla \cdot \mathbf{q}\|^2.\non
\eea
From Remark \ref{AH}(2), taking $M_0=\langle \chi\rangle$, then we have
\bea
&&\int_\Omega f(\chi)\tch dx \geq K_1 \int_\Omega F(\chi)dx+K_2\|\tch\|^2-K_3,\non\\
&&\int_\Gamma g(\xi)\txi dS\geq K_1'\int_\Gamma G(\xi) dS+K_2'\|\txi\|_{H_\Gamma}-K_3',\non
\eea
where $K_i>0, K_i'>0$  $(i=1,2,3)$ are independent of $\chi, \xi$. Besides, from Remark \ref{AH}(3) it follows
\be
Q_1\left(\int_\Omega f(\chi) dx+\int_\Gamma g(\xi) dS\right)
\leq \frac{\kappa_1 K_1}{2} \int_\Omega F(\chi)dx + \frac{\kappa_1 K_1'}{2}\int_\Gamma G(\xi) dS+ K_4|Q_1|,\non
\ee
where $K_4$ depends on $K_1, K_1', \kappa_1$.
By the Poincar\'e inequality, there exists $C_P>0$ depending on $\Omega$ such that
\bea
&& -Q_1\int_\Gamma \txi_t dS + \kappa_1\int_\Omega \tth\tch dx-\kappa_1 Q_1\int_\Gamma \txi dS\non\\
&\leq&  |Q_1||\Gamma|^\frac12\|\txi_t\|_{H_\Gamma}+\kappa_1C_P\|\tth\|\|\nabla\tch\|+ \kappa_1|Q_1||\Gamma|^\frac12\|\txi\|_{H_\Gamma}\non\\
&\leq& \frac12\|\txi_t\|_{H_\Gamma}^2+\frac{\kappa_1K_2'}{2}\|\txi\|_{H_\Gamma}^2+\frac{\kappa_1}{2}\|\nabla \tch\|^2
+\frac{\kappa_1C_P^2}{2}\|\tth\|^2+K_5|Q_1|^2,\non
\eea
where $K_5$ depends on $\langle \chi_1\rangle, |\Gamma|$ and $\kappa_1$.
Next, there exists some $C_\Omega>0$ depending on $\Omega$ such that
\bea
&&-\kappa_2\int_\Omega \mathbf{q}\cdot \nabla A_0^{-1} \tth dx -\kappa_2\int_\Omega \mathbf{q}\cdot \nabla A_0^{-1}\tch_tdx
+\kappa_2\|A_0^{-\frac12}\nabla \cdot \mathbf{q}\|^2\non\\
&\leq& \kappa_2C_\Omega(\|\mathbf{q}\|\|\tth\|+ \|\mathbf{q}\|\|A_0^{-\frac12}\tch_t\|+\|\mathbf{q}\|^2)\non\\
&\leq& \frac{\kappa_2}{2}\|\tth\|^2+\frac{\kappa_2C_\Omega}{2} \|A_0^{-\frac12}\tch_t\|+\frac{\kappa_2}{2}C_\Omega(C_\Omega+3)\|\mathbf{q}\|^2.\label{kk}
\eea
We now choose $\kappa_1, \kappa_2>0$ sufficiently small so that
\be
\kappa_1\leq\frac14, \quad \kappa_2C_\Omega\leq \frac12,\quad
\kappa_1+\frac{\kappa_2C_\Omega}{2}\leq \frac12,\quad \frac{\kappa_2}{2}C_\Omega(C_\Omega+3)\leq \frac12,\quad
\frac{\kappa_1C_P^2}{2}\leq \frac{\kappa_2}{4}.\non
\ee
From the above estimates we deduce the following inequality
\be
\frac{d}{dt} \mathcal{Y}(t)+ \mathcal{I}(t)\leq K_6(1+e^{-2t}), \quad \forall \, t \geq 0\label{Y}
\ee
where
\bea
\mathcal{Y}&=& \frac12\|\tth\|^2+\frac{1}{2}\|\tq\|^2
+\frac{1}{2}\|A_0^{-\frac12}\tch_t\|^2+ \frac12\|\nabla \tch\|^2+\int_\Omega F(\chi)dx  \non\\
&&  +\frac12\|\nabla_\Gamma \txi\|_{H_\Gamma}^2+\frac{\kappa_1}{2}\|\txi\|_{H_\Gamma}^2+ \int_\Gamma G(\xi)dS
+\kappa_1\int_\Omega A_0^{-\frac12}\tch_t A_0^{-\frac12}\tch dx\non\\
&& +\frac{\kappa_1}{2}\| A_0^{-\frac12}\tch\|^2+\frac{\kappa_1\alpha}{2}\|\tch\|^2+ \kappa_2\int_\Omega \mathbf{q}\cdot \nabla A_0^{-1}\tth dx\label{Ya}
\eea
and
\bea
\mathcal{I}&=&\frac{\kappa_2}{4}\|\tth\|^2+\frac12\|\tq\|^2+\frac12\|A_0^{-\frac12}\tch_t\|^2+\alpha \|\tch_t\|^2+\frac12\|\txi_t\|^2_{H_\Gamma}
 +\frac{\kappa_1}{2}\|\nabla \tch\|^2\non\\
&& +\kappa_1\|\nabla_\Gamma \txi\|^2_{H_\Gamma}+\frac{\kappa_1K_2'}{2}\|\txi\|^2_{H_\Gamma}+\kappa_1K_1 \int_\Omega F(\chi)dx
+\kappa_1K_1'\int_\Gamma G(\xi) dS.\non
\eea
By comparison, it is easy to verify that
\be
\mathcal{Y}(t)\leq K_{7} \mathcal{I}(t), \quad \forall \, t \geq 0,\non
\ee
which, together with \eqref{Y}, implies
\be
\frac{d}{dt} \mathcal{Y}(t)+ K_{8}\mathcal{Y}(t)\leq K_6(1+e^{-t})\leq 2K_6,\quad \forall \, t \geq 0.\non
\ee
As a result,
\be
\mathcal{Y}(t)\leq \mathcal{Y}(0)e^{-K_{8}t}+\frac{2K_6}{K_{8}},\quad \forall\, t\geq 0.\label{Y1}
\ee
Since the quantities $\langle\theta\rangle$, $\langle\chi\rangle$ and $\langle\chi_t\rangle$ are uniformly bounded in time, we can deduce that
\be
\mathcal{Y} + K_{9}
 \geq  K_{10}\left(\|\tth\|^2+\|\tq\|^2+\|\nabla \tch\|^2+ \|\tch\|^2+\|\nabla_\Gamma \txi\|_{H_\Gamma}^2+\|\txi\|_{H_\Gamma}^2
+\|A_0^{-\frac12}\tch_t\|^2\right),\label{Y2}
\ee
where the constants $K_{9}, K_{10}>0$ may depend on $\Omega$, $\langle\theta_0\rangle$, $\langle\chi_0\rangle$ and $\langle\chi_1\rangle$.
Then, from \eqref{Y1} and \eqref{Y2} one infers estimate \eqref{disa1}. The proof is complete.
\end{proof}

\noindent \textbf{Higher-order estimates}. In what follows, we derive the uniform-in-time estimate in the higher-order space $\mathbb{Y}$.
For the sake of simplicity, from now on we shall indicate by $C$ or $C_i, \, i \in \mathbb{N}$, a positive constant that may vary from line to
line and also in the same line.
\bl
\label{esY}
Let the assumptions (H1)--(H3) be satisfied. Suppose $(\theta(t), \mathbf{q}(t), \chi(t), \xi(t), \chi_t(t))$ is a weak solution of
system \eqref{c1}--\eqref{c7} with initial data $(\theta_0, \tq_0, \chi_0, \xi_0, \chi_1) \in \mathbb{Y}$. Then we have
\be
\|(\theta(t), \mathbf{q}(t), \chi(t), \xi(t), \chi_t(t))\|_{\mathbb{Y}}\leq C(\|(\theta_0, \mathbf{q}_0, \chi_0, \xi_0, \chi_1)\|_{\mathbb{Y}}, \quad \forall\, t\geq 0. \label{eesY}
\ee
\el
\begin{proof}
 We (formally) differentiate \eqref{c1a}--\eqref{c5a} with respect to time and we get
 \begin{align}
 &(\tth+\tch)_{tt}+\nabla \cdot \tq_t=0, \quad \text{in \, } \Omega\times (0,\infty),\label{dc1a}\\
 & \tq_{tt}+\tq_t=-\nabla \tth_t, \quad \text{in \, } \Omega\times (0,\infty), \label{dc2a}\\
 & \tch_{ttt}+\tch_{tt}-\Delta\tilde{\mu}_t =0, \quad \text{in \, } \Omega\times (0,\infty),\label{dc3a}\\
 & \tilde{\mu}_t=-\Delta\tch_t+\alpha \tch_{tt}+f'(\chi)\chi_t-\tth_t,  \quad \text{in \, } \Omega\times (0,\infty), \label{dc4a}\\
 &  \txi_{tt}-\Delta_\Gamma \txi_t +g'(\xi)\xi_t+\partial_\nu \tch_t-Q_1(t)=0, \quad \text{on \, } \Gamma\times (0,\infty), \label{dc6a}\\
 & \tq_t\cdot \nu =\partial_\nu\tilde{\mu}_t=0, \quad \text{on \, } \Gamma\times (0,\infty),\label{dc5a}\\
 &\tth_t(0)= - \chi_1+ \langle \chi_1 \rangle-\nabla \tq_0, \quad \tq_t(0)=-\tq_0-\nabla \theta_0,,\quad \text{in}\ \Omega, \\
 & \tch_t(0)=\chi_1-\langle \chi_1 \rangle, \quad \tch_{tt}(0)=-\chi_1+\langle \chi_1 \rangle+\Delta(-\Delta \chi_0+\alpha \chi_1+f(\chi_0)-\theta_0)
 \\
 &  \txi_t(0)=\Delta_\Gamma \xi_0 - g(\xi_0)-\partial_\nu \chi_0 - \langle \chi_1 \rangle,\quad \text{in}\ \Omega.\hskip2truecm \label{dc7a}
  \end{align}
It is easy to verify that the initial datum can be controlled as follows
\be
\|(\tth_t(0), \tq_t(0), \tch_t(0), \txi_t(0), \tch_{tt}(0)\|_{\mathbb{X}}\leq C \|(\theta_0, \tq_0, \chi_0, \xi_0, \chi_1)\|_{\mathbb{Y}},
\ee
where $C$ is a constant depending on $\Omega$ and $\Gamma$.

Multiplying \eqref{dc1a} and \eqref{dc2a} by $\tth_t$ and $\tq_t$, respectively, and then integrating over $\Omega$, we obtain
\bea
&&\frac12\frac{d}{dt}\|\tth_t\|^2-\int_\Omega \tq_t\cdot \nabla\tth_t dx =-\int_\Omega \tch_{tt}\tth_t dx,\label{de1}\\
&&\frac{1}{2}\frac{d}{dt}\|\tq_t\|^2 +\|\tq_t\|^2+\int_\Omega \tq_t \cdot \nabla \tth_t dx =0,\label{de2}
\eea
Multiplying \eqref{dc3a} by $A_0^{-1}\tch_{tt}$ and integrating over $\Omega$, we have
\bea
&&\frac12\frac{d}{dt}\left[\|A_0^{-\frac12}\tch_{tt}\|^2 + \|\nabla \tch_t\|^2
+ \int_\Omega (f'(\chi)+L)\tch_t^2 dx+ \|\nabla_\Gamma \txi_t\|^2_{H_\Gamma}+\int_\Gamma (g'(\xi)+L)\txi_t^2 dS\right]\non\\
&& +\|A_0^{-\frac12}\tch_{tt}\|^2+\alpha \|\tch_{tt}\|^2+\|\txi_{tt}\|^2_{H_\Gamma}\non\\
&=& \int_\Omega\tth_t\tch_{tt}dx+\frac12 \int_\Omega f''(\chi)\chi_t\tch_t^2 dx -Q_1\int_\Omega f'(\chi) \tch_{tt} dx
+\frac12 \int_\Gamma g''(\xi)\xi_t\txi_t^2 dS\non\\
&&  -Q_1\int_\Gamma g'(\xi) \txi_{tt} dS +L\int_\Omega \tch_t\tch_{tt}+ L\int_\Gamma \txi_t\txi_{tt}dS + Q_1\int_\Gamma \txi_{tt} dS,\label{de3}
\eea
where we have used the identities
\bea
&& \int_\Omega f'(\chi)\chi_t \tch_{tt} dx=\frac12\frac{d}{dt} \int_\Omega f'(\chi)\tch_t^2 dx-\frac12 \int_\Omega f''(\chi)\chi_t\tch_t^2 dx
+Q_1\int_\Omega f'(\chi) \tch_{tt} dx,\non\\
&& \int_\Gamma g'(\xi)\xi_t \txi_{tt} dS=\frac12\frac{d}{dt} \int_\Gamma g'(\xi)\txi_t^2 dS-\frac12 \int_\Gamma g''(\xi)\xi_t\txi_t^2 dS
+Q_1\int_\Gamma g'(\xi) \txi_{tt} dS.\non
\eea
Here $L\geq c_0+1$ is a positive constant such that $f'(y)+L\geq 1$ and $g'(y)+L\geq 1$ (cf. Remark \ref{AH}).
On the other hand, multiplying \eqref{dc3a} by $ A_0^{-1}\tch_t$ and integrating over $\Omega$, we get
\bea
&& \frac{d}{dt}\left(\int_\Omega A_0^{-\frac12}\tch_{tt} A_0^{-\frac12}\tch_t dx
+ \frac12\| A_0^{-\frac12}\tch_t\|^2+\frac{\alpha}{2}\|\tch_t\|^2+\frac{1}{2}\|\txi_t\|^2_{H_\Gamma}\right)\non\\
&& -\| A_0^{-\frac12}\tch_{tt}\|^2+\|\nabla  \tch_t\|^2+\|\nabla_\Gamma\txi_t\|^2_{H_\Gamma}
+\int_\Omega (f'(\chi)+L)\tch_t^2dx+\int_\Gamma (g'(\xi)+L)\txi_t^2 dS\non\\
&=& -Q_1\int_\Omega f'(\chi) \tch_tdx-Q_1 \int_\Gamma g'(\xi) \txi_t dS +\int_\Omega \tth_t\tch_t dx+Q_1\int_\Gamma \txi_t dS\non\\
&&+ L(\|\tch_t\|^2+\|\txi_t\|_{H_\Gamma}^2).\label{de4}
\eea
Finally, using the equations \eqref{c1a} and \eqref{c2a} we obtain
\bea
&& \frac{d}{dt}\int_\Omega \mathbf{q}_t\cdot \nabla A_0^{-1}\tth_t dx+\|\tth_t\|^2 \non\\
&=&\int_\Omega \mathbf{q}_{tt} \cdot  \nabla A_0^{-1}\tth_t dx+ \int_\Omega \mathbf{q}_t \cdot   \nabla A_0^{-1}\tth_{tt} dx+\|\tth_t\|^2\non\\
&=& -\int_\Omega \mathbf{q}_t\cdot \nabla A_0^{-1} \tth_t dx -\int_\Omega \mathbf{q}_t\cdot \nabla A_0^{-1}\tch_{tt}dx
+\|A_0^{-\frac12}\nabla \cdot \mathbf{q}_t\|^2.\label{de5}
\eea
Multiplying \eqref{de4}, \eqref{de5} by some small constants $\kappa_3, \kappa_4>0$ (to be chosen later), respectively, and adding the resultants
with \eqref{de1}--\eqref{de3}, one deduces that
\be
\frac{d}{dt} \mathcal{Y}_1(t)+\mathcal{I}_1(t)\leq \mathcal{R}_1(t),\quad \forall \, t\geq 0,\label{dishigh}
\ee
where
\bea
\mathcal{Y}_1&=& \frac12\|\tth_t\|^2+\frac12\|\tq_t\|^2+\frac{1}{2}\|A_0^{-\frac12}\tch_{tt}\|^2 + \frac12\|\nabla \tch_t\|^2
+ \frac12\|\nabla_\Gamma \txi_t\|^2_{H_\Gamma}\non\\
&& + \frac12\int_\Omega (f'(\chi)+L)\tch_t^2 dx+ \frac12 \int_\Gamma (g'(\xi)+L)\txi_t^2 dS
-\kappa_3 \int_\Omega A_0^{-\frac12}\tch_{tt} A_0^{-\frac12}\tch_t dx\non\\
&&
+ \frac{\kappa_3}{2}\| A_0^{-\frac12}\tch_t\|^2+\frac{\kappa_3\alpha}{2}\|\tch_t\|^2+\frac{\kappa_3}{2}\|\txi_t\|^2_{H_\Gamma}\non\\
&& +\kappa_4\int_\Omega \mathbf{q}_t\cdot \nabla A_0^{-1}\tth_t dx,\label{y1}
\eea
\bea
\mathcal{I}_1&=& \|\tq_t\|^2+(1-\kappa_3)\|A_0^{-\frac12}\tch_{tt}\|^2+\alpha \|\tch_{tt}\|^2+\|\txi_{tt}\|^2_{H_\Gamma}\non\\
&& +\kappa_3\|\nabla \tch_t\|^2+\kappa_3\|\nabla_\Gamma\txi_t\|^2_{H_\Gamma} +\kappa_3\int_\Omega (f'(\chi)+L)\tch_t^2dx\non\\
&& +\kappa_3\int_\Gamma (g'(\xi)+L)\txi_t^2 dS+ \kappa_4\|\tth_t\|^2+\kappa_4\int_\Omega \mathbf{q}_t\cdot \nabla A_0^{-1} \tth_t dx \non\\
&& +\kappa_4\int_\Omega \mathbf{q}_t\cdot \nabla A_0^{-1}\tch_{tt}dx
-\kappa_4\|A_0^{-\frac12}\nabla \cdot \mathbf{q}_t\|^2\label{i1}
\eea
and
\bea
\mathcal{R}_1 &=& \frac12 \int_\Omega f''(\chi)\chi_t\tch_t^2 dx -Q_1\int_\Omega f'(\chi) \tch_{tt} dx+\frac12 \int_\Gamma g''(\xi)\xi_t\txi_t^2 dS\non\\
&&  -Q_1\int_\Gamma g'(\xi) \txi_{tt} dS +L\int_\Omega \tch_t\tch_{tt}+ L\int_\Gamma \txi_t\txi_{tt}dS + Q_1\int_\Gamma \txi_{tt} dS\non\\
&&  -\kappa_3 Q_1\int_\Omega f'(\chi) \tch_tdx-\kappa_3 Q_1 \int_\Gamma g'(\xi) \txi_t dS  +\kappa_3\int_\Omega \tth_t\tch_t dx\non\\
&& +\kappa_3 Q_1\int_\Gamma \txi_t dS+\kappa_3 L(\|\tch_t\|^2+\|\txi_t\|_{H_\Gamma}^2).\label{R1}
\eea
Using the H\"older inequality and choosing a suitable $L$, for $\kappa_3, \kappa_4$ sufficiently small, we find
\bea
\mathcal{Y}_1&\geq  & C_1\|(\tth_t, \tq_t, \tch_t, \txi_t, \tch_{tt})\|_{\mathbb{X}}^2,\label{esy1}\\
\mathcal{I}_1&\geq & C_2\mathcal{Y}_1+\alpha \|\tch_{tt}\|^2+\|\txi_{tt}\|^2_{H_\Gamma},\label{esy1i}
\eea
where the constant $C_1, C_2$ may depend on $\Omega$, $\Gamma$, $\alpha$, $\kappa_3$, $\kappa_4$ and $L$.

Next, we estimate the reminder term $\mathcal{R}_1$. From H\"older inequality, Young's inequality, the Sobolev embedding theorem and
the growth assumptions (H3) on $f$ and $g$, it follows
\bea
\mathcal{R}_1&\leq& \|f''(\chi)\|_{L^6(\Omega)}\|\chi_t\|\|\tch_t\|^2_{L^6(\Omega)}+|Q_1|\|f'(\chi)\|\|\tch_{tt}\|
+\|g''(\xi)\|_{L^6(\Gamma)}\|\xi_t\|\|\txi_t\|^2_{L^6(\Gamma)}\non\\
&&+|Q_1|\|g'(\xi)\|\|\txi_{tt}\|+ L(\|\tch_t\|\|\tch_{tt}\|+\|\txi_t\|_{H_\Gamma}\|\txi_{tt}\|_{H_\Gamma})
+|Q_1||\Gamma|^\frac12\|\txi_{tt}\|_{H_\Gamma}\non\\
&& +\kappa_3|Q_1|(\|f'(\chi)\|\|\tch_t\|+\|g'(\xi)\|_{H_\Gamma}\|\txi_t\|_{H_\Gamma})+\kappa_3\|\tth_t\|\|\tch_t\|\non\\
&& +\kappa_3|Q_1||\Gamma|^\frac12\|\txi_t\|_{H_\Gamma}+\kappa_3 L(\|\tch_t\|^2+\|\txi_t\|_{H_\Gamma}^2)\non\\
&\leq& \epsilon (\|\tch_t\|^2_{V}+\|\txi_t\|^2_{V_\Gamma}+\|\tch_{tt}\|^2+\|\txi_{tt}\|^2_{H_\Gamma}+\|\tth_t\|^2)\non\\
&& + \mathcal{Q}_1(\|(\chi, \xi)\|_{\mathbb{H}^1})(\|\chi_t\|^2\|\tch_t\|^2_{V}+\|\xi_t\|^2\|\txi_t\|_{V_\Gamma}^2)\non\\
&& +\mathcal{Q}_2(\|(\chi, \xi)\|_{\mathbb{H}^1})|Q_1|^2+C(\|\tch_t\|^2+\|\txi_t\|_{H_\Gamma}^2),\label{R1e}
\eea
where $\mathcal{Q}_1, \mathcal{Q}_2$ are certain monotone increasing functions.
Taking $\epsilon$ sufficiently small, from the above estimates  \eqref{dishigh}--\eqref{R1e} we infer
\bea
\frac{d}{dt}\mathcal{Y}_1 + C_3\mathcal{Y}_1
&\leq &  \mathcal{Q}_1(\|(\chi, \xi)\|_{\mathbb{H}^1})(\|\chi_t\|^2+\|\xi_t\|^2)\mathcal{Y}_1\non\\
&& +\mathcal{Q}_2(\|(\chi, \xi)\|_{\mathbb{H}^1})|Q_1|^2+C(\|\tch_t\|^2+\|\txi_t\|_{H_\Gamma}^2),\label{dify1}
\eea
where $C_3$ is a small constant that may depend on  $\Omega$, $\Gamma$, $\alpha$, $\kappa_3$, $\kappa_4$ and $L$, but not on the solution.
Besides, from the integrability of estimate \eqref{kkk} (see Section 5) we easily see that
\be
\int_0^\infty \|\tch_t\|^2+\|\txi_t\|_{H_\Gamma}^2 dt <+\infty.\label{kka}
\ee
Using the dissipative estimate \eqref{disa1} (so that $\mathcal{Q}_1(\|(\chi, \xi)\|_{\mathbb{H}^1}), \mathcal{Q}_2(\|(\chi, \xi)\|_{\mathbb{H}^1})$
are uniformly bounded for all time) and the Gronwall-type lemma (see e.g., \cite[Lemma 2.2]{GP}), then from \eqref{dify1} we infer
\be
\mathcal{Y}_1(t)\leq C_4\mathcal{Y}_1(0)e^{-\frac{C_3t}{2}}+C_5, \quad \forall\, t\geq 0,\non
\ee
where $C_4, C_5$ may depend on $\|(\theta_0, \tq_0, \chi_0, \xi_0, \chi_1)\|_{\mathbb{X}}$. Then, by the definition of $\mathcal{Y}_1$ and \eqref{esy1}, we have
\be
\|(\tth_t(t), \tq_t(t), \tch_t(t), \txi_t(t), \tch_{tt}(t))\|_{\mathbb{X}}\leq C(\|(\theta_0, \tq_0, \chi_0, \xi_0, \chi_1)\|_{\mathbb{Y}}),
\quad \forall\, t\geq 0,\label{unies1}
\ee
which also easily yields that
\be
\|(\theta_t(t), \tq_t(t), \chi_t(t), \xi_t(t), \chi_{tt}(t))\|_{\mathbb{X}}\leq C(\|(\theta_0, \tq_0, \chi_0, \xi_0, \chi_1)\|_{\mathbb{Y}}),
\quad \forall\, t\geq 0.\label{unies2}
\ee
Using the estimate \eqref{unies1}, from the equations \eqref{c1a}--\eqref{c2a} we deduce
\be
\|\nabla \theta(t)\|\leq \|\mathbf{q}_t(t)\|+\|\mathbf{q}(t)\|\leq C, \quad \|\nabla \cdot \mathbf{q}(t)\|\leq \|\tth_t(t)\|+\|\tch_t(t)\|\leq C.
\quad \forall \, t \geq 0,\non
\ee
Applying the curl operator to \eqref{c2a}, we have
\bea
&& (\nabla \times\mathbf{q})_t(t)+(\nabla \times\mathbf{q})(t)=0, \quad \forall\, t\geq 0 \non\\
&& (\nabla \times\mathbf{q})(0)=\nabla \times\mathbf{q}_0, \non
\eea
so that $\|(\nabla \times\mathbf{q})(t)\|\leq \|\nabla \times\mathbf{q}_0\|$, for all $t\geq 0$.
Combining the above estimates and \eqref{disa1}, we get
\be
\|\theta(t)\|_{V}\leq C, \quad \|\mathbf{q}(t)\|_{\mathbf{H}^1(\Omega)}\leq C, \quad \forall\, t\geq 0.\non
\ee
 It remains to prove the estimate of $(\chi, \xi)$ in $\mathbb{H}^3$. To this purpose, we rewrite \eqref{c4} and \eqref{c6} as follows
\bea
&& -\Delta \chi=\mu-f(\chi)-\alpha \chi_t+\theta:=h_1,\label{eep1}\\
&&  -\Delta_\Gamma \xi+\partial_\nu \chi+\beta \xi=- \xi_t - g(\xi)+ \beta \xi:=h_2,\label{eep2}
\eea
where $\beta>0$ is a positive constant.
Since $\mu$ satisfies \eqref{c3} and \eqref{c5}, so that
\be
-\Delta \mu= -(\chi_{tt}+\chi_t), \ \  \text{in} \ \Omega \times (0,\infty), \quad \quad \partial_\nu\mu=0, \ \ \text{on} \ \Gamma \times (0,\infty),\non
\ee
then from estimate \eqref{unies2} we infer
\bea
\|\mu\|_{V}&\leq& C(\|\chi_{tt}+\chi_t\|_{V^*}+|\langle\mu\rangle|\non\\
&\leq& C+| \langle f(\chi)\rangle|+\alpha |\langle \chi_t\rangle|+|\langle \theta \rangle|
+ \frac{1}{|\Omega|} \left(   \left|\int_\Gamma \xi_t dS\right|+ \left|\int_\Gamma g(\xi)dS\right|
\right)\non\\
&\leq& C.\non
\eea
As a result, we have
\be \|h_1(t)\|\leq C, \quad \|h_2(t)\|_{H_\Gamma}\leq C, \quad \forall \, t\geq 0.\non
\ee
Now we apply the regularity theorem \cite[Lemma A.1]{MZ05} to the elliptic problem \eqref{ep1}--\eqref{ep2} (see Section 4),
obtaining
\be
\|(\chi(t), \xi(t))\|_{\mathbb{H}^2}\leq C\|(h_1(t), h_2(t))\|_{\mathbb{H}}\leq C, \quad \forall\, t\geq 0.\non
\ee
Then, from the above estimate and Sobolev embedding theorem we infer
\be \|h_1(t)\|_{V}\leq C, \quad \|h_2(t)\|_{V_\Gamma}\leq C, \quad \forall \, t\geq 0.\non
\ee
An application of the higher-order regularity theorem \cite[Corollary A.1]{MZ05} to the elliptic problem \eqref{ep1}--\eqref{ep2} yields
\be
\|(\chi(t), \xi(t))\|_{\mathbb{H}^3}\leq C\|(h_1(t), h_2(t))\|_{\mathbb{H}^1}\leq C, \quad \forall\, t\geq 0.\non
\ee
Collecting all the above estimates, we have shown that $(\theta(t), \mathbf{q}(t), \chi(t), \xi(t), \chi_{t})$ is uniformly bounded in $\mathbb{Y}$
and the proof is complete.
\end{proof}

\subsection{Existence and uniqueness}

Based on the uniform estimates obtained in the previous section, we are able to prove the existence and uniqueness of suitable solutions to problem
\eqref{c1}--\eqref{c7}.

\begin{theorem}\label{exe}
 Suppose that assumptions (H1)--(H3) are satisfied. Then we have

  (i) For any initial datum
 $(\theta_0, \mathbf{q}_0, \chi_0, \xi_0, \chi_1)\in \mathbb{Y}$, problem \eqref{c1}--\eqref{c7} admits a unique weak solution.

  (ii) For any initial datum
 $(\theta_0, \mathbf{q}_0, \chi_0, \xi_0, \chi_1)\in \mathbb{X}$, problem \eqref{c1}--\eqref{c7} admits a unique energy solution.
 \end{theorem}
 \begin{proof}
 (i) Based on the uniform dissipative estimate \eqref{eesY}, it is standard to prove the existence of  global weak solutions to problem
 \eqref{c1}--\eqref{c7} by using a Faedo--Galerkin scheme as in \cite{GMS09,GMS} for the Cahn-Hilliard equation subject to dynamic
 boundary conditions. The details are omitted here.

 Concerning the uniqueness, it suffices to show the continuous dependence estimate for two solutions
 $(\theta^{(i)}, \mathbf{q}^{(i)}, \chi^{(i)}, \xi^{(i)}, \chi_t^{(i)})$ corresponding to the two sets of data
 $(\theta_0^{(i)}, \mathbf{q}_0^{(i)}, \chi_0^{(i)}, \xi_0^{(i)}, \chi_1^{(i)})$ ($i=1,2$).
 For this purpose, we write down the system for
 $$(\bar\theta, \bar{\mathbf{q}}, \bar{\chi}, \bar \xi, \bar{\chi}_t)=(\tth^{(1)}, \mathbf{q}^{(1)}, \tch^{(1)}, \txi^{(1)}, \tch_t^{(1)})
 -(\tth^{(2)}, \mathbf{q}^{(2)}, \tch^{(2)}, \txi^{(2)}, \tch_t^{(2)}),$$
 such that
 \bea
 &&(\bar\theta+\bar\chi)_t+\nabla \cdot \bar\tq=0,\quad \text{in \, } \Omega\times (0,\infty),\label{dc1}\\
 && \bar{\tq}_t+\bar{\tq}=-\nabla \bar\theta,\quad \text{in \, } \Omega\times (0,\infty),\label{dc2}\\
 &&  \bar{\chi}_{tt}+\bar{\chi}_t-\Delta(\tilde{\mu}^{(1)}-\tilde{\mu}^{(2)})=0,\quad \text{in \, } \Omega\times (0,\infty),\label{dc3}\\
 && \tilde{\mu}^{(1)}-\tilde{\mu}^{(2)}=-\Delta\bar \chi
 +\alpha \bar{\chi}_t+f(\chi^{(1)})-f(\chi^{(2)})-\bar{\theta}, \quad \text{in \, } \Omega\times (0,\infty),\label{dc4}\\
 && \bar{\tq}\cdot \nu =\partial_\nu\tilde{\mu}^{(i)}=0, \ \ i=1,2,\qquad \text{on \, } \Gamma\times (0,\infty),\label{dc5}\\
 &&  \bar{\xi}_t-\Delta_\Gamma \bar{\xi}+\partial_\nu \bar \chi+g(\xi^{(1)})-g(\xi^{(2)})+\bar{Q}_1=0,\quad  \text{on \, }\Gamma\times (0,\infty),\label{dc6}
\eea
 where $\bar{Q}_1=(\langle \chi_1^{(1)}\rangle-\langle\chi_1^{(2)}\rangle)e^{-t}.$

 The regularity of weak solutions allows us to multiply \eqref{dc1} by $\bar\theta$, \eqref{dc2} by $\bar{\mathbf{q}}$ and \eqref{dc3} by $A_0^{-1}\bar{\chi}_t$,
 respectively, and then integrate over $\Omega$. Adding the resulting equations together, we have
\bea
&& \frac{d}{dt}\left(\frac12\|\bar\theta\|^2+\frac12\|\mathbf{q}\|^2+\frac12\|A_0^{-\frac12}\bar{\chi}_t\|^2+\frac12\|\nabla \bar\chi\|^2
+\frac12\|\nabla_\Gamma \bar{\xi}\|^2\right)\non\\
&& +\|\bar{\mathbf{q}}\|^2+\| A_0^{-\frac12} \bar{\chi}_t\|^2+ \alpha\|\bar{\chi}_t\|^2+\|\bar{\xi}_t\|^2_{H_\Gamma}  \non\\
&=& -\bar{Q}_1\int_\Gamma \bar{u}_t dS-\int_\Omega (f(\chi^{(1)})-f(\chi^{(2)})) \bar{\chi}_t dx
-\int_\Gamma (g(\xi^{(1)})-g(\xi^{(2)}))\bar{\xi}_t dS.\label{diff}
\eea
Using the uniform estimates \eqref{disa1} for the two solutions, the growth assumption (H3) and Sobolev embedding theorems, we infer
\bea
-\int_\Omega (f(\chi^{(1)})-f(\chi^{(2)})) \bar{\chi}_t dx&\leq& \|f'\|_{L^\frac{2(p+2)}{p}(\Omega)}\|\bar{\chi}\|_{L^{p+2}(\Omega)}\|\bar{\chi}_t\|\non\\
&\leq&  \frac{\alpha}{2}\| \bar{\chi}_t\|^2+ C\|\bar{\chi}\|_{V}^2\non\\
&\leq& \frac{\alpha}{2}\| \bar{\chi}_t\|^2+ C\|\nabla \bar{\chi}\|^2.\non
\eea
Similarly, we have
\bea
-\int_\Gamma (g(\xi^{(1)})-g(\xi^{(2)}))\bar{\xi}_t dS&\leq& \frac12\|\bar{\xi}_t\|_{H_\Gamma}^2+C\|\bar{\xi}\|_{V_\Gamma}^2\non\\
&\leq& \frac12\|\bar{\xi}_t\|_{H_\Gamma}^2+C\|\nabla_\Gamma \bar \xi\|_{H_\Gamma}^2+ C\|\nabla \bar{\chi}\|^2.\non
\eea
Moreover, by H\"older inequality and Young inequality, we obtain
\be
-\bar{Q}_1\int_\Gamma \bar{\xi}_t dS\leq \frac12\|\bar{\xi}_t\|^2_{H_\Gamma} +C(\bar{Q}_1)^2.\non
\ee
Therefore, we find
\bea
&& \frac{d}{dt}\left(\frac12\|\bar\theta\|^2+\frac12\|\mathbf{q}\|^2+\frac12\|A_0^{-\frac12}\bar{\chi}_t\|^2+\frac12\|\nabla \bar\chi\|^2
+\frac12\|\nabla_\Gamma \bar{\xi}\|^2\right)\non\\
&\leq& C\|\nabla \bar{\chi}\|^2+C\|\nabla_\Gamma \bar \xi\|_{H_\Gamma}^2+C(\bar{Q}_1)^2.\non
\eea
Then, by the Gronwall lemma and the conservation properties \eqref{cv1}--\eqref{cv2}, we can deduce
\bea
&\|((\theta^{(1)}-\theta^{(2)})(t), (\mathbf{q}^{(1)}-\mathbf{q}^{(2)})(t),
(\chi^{(1)}-\chi^{(2)})(t), (\xi^{(1)}- \xi^{(2)})(t), (\chi_t^{(1)}-\chi_t^{(2)})(t))\|_{\mathbb{X}}\non\\
&\leq C_1e^{C_2T}\|(\theta_0^{(1)}-\theta_0^{(2)}, \mathbf{q}^{(1)}_0-\mathbf{q}^{(2)}_0,
\chi^{(1)}_0-\chi^{(2)}_0, \xi^{(1)}_0- \xi^{(2)}_0, \chi^{(1)}_1-\chi^{(2)}_1)\|_{\mathbb{X}},\label{conti}
\eea
for any $t \in [0, T]$, where $C_1, C_2$ only depend on the $\mathbb{X}$-norms of the initial data, $\alpha$, $|\Omega|$ and $|\Gamma|$.
This completes the proof for uniqueness.

(ii) We note that in the continuous dependence estimate for weak solutions \eqref{conti}, the constant $C_1, C_2$ only depend on the $\mathbb{X}$-norms
of the initial data. This fact enables us to prove the existence and uniqueness of energy solutions to problem \eqref{c1}--\eqref{c7} by
using the standard density argument. The details are left to the interested reader.
\end{proof}

 A straightforward consequence of the above result yields
\bc
Suppose that assumptions (H1)--(H3) are satisfied.

(i) For any initial datum $(\theta_0, \mathbf{q}_0, \chi_0, \xi_0, \chi_1) \in \mathbb{Y}$, the unique global weak solution to problem
\eqref{c1}--\eqref{c7} defines a semigroup $S_1(t): \mathbb{Y}\to \mathbb{Y}$ such that
$$S_1(t)(\theta_0, \mathbf{q}_0, \chi_0, \xi_0, \chi_1)=(\theta(t), \mathbf{q}(t), \chi(t), \xi(t), \chi_t(t)), \quad \forall \, t\geq 0.$$

(ii) For any initial datum $(\theta_0, \mathbf{q}_0, \chi_0, \xi_0, \chi_1) \in \mathbb{X}$, the unique global energy solution to problem
\eqref{c1}--\eqref{c7} defines a strongly continuous semigroup $S_2(t): \mathbb{X}\to \mathbb{X}$ such that
$$S_2(t)(\theta_0, \mathbf{q}_0, \chi_0, \xi_0, \chi_1)=(\theta(t), \mathbf{q}(t), \chi(t), \xi(t), \chi_t(t)), \quad \forall \, t\geq 0.$$
\ec

 \br The estimate \eqref{conti} provides a continuous dependence result in the (lower) $\mathbb{X}$-norm. As a consequence, $S_1(t)$ turns out
 to be a closed semigroup in the sense of \cite{PZ07}.
 \er


\section{Global attractor for energy solutions}\setcounter{equation}{0}
\label{sec4}

In this section, we study the associated infinite-dimensional dynamical system defined by the semigroup $S_2(t)$ on $\mathbb{X}$. More precisely, we will prove that $S_2(t)$ possesses the global attractor in the phase space
$$\mathbb{X}_{M, M'}=\{(z_1, \mathbf{z}_2, z_3,z_4,z_5)\in \mathbb{X}:\ |\langle z_1+z_3\rangle|\leq M,\  |\langle z_3+z_5\rangle|\leq M, \ |z_5|\leq M'\},$$
endowed with the metric induced by the norm on $\mathbb{X}$. Here $M, M'\geq 0$ are arbitrary constants.
We note that the choice of the phase space is due to the constraints \eqref{cv1}, \eqref{cv2} and the decay property \eqref{cv3}.

We now state the main result of this section.
\begin{theorem}\label{GA}
Suppose that (H1)--(H3) are satisfied. The semigroup $S_2(t)$ defined by the global energy solutions to problem \eqref{c1}--\eqref{c7}
on $\mathbb{X}_{M,M'}$ possesses a compact connected global attractor $\mathcal{A}$ which is bounded in $\mathbb{Y}$.
 \end{theorem}

The proof of Theorem \ref{GA} consists of several steps. First, we show that the restriction of $S_2(t)$ on $\mathbb{X}_{M, M'}$ admits a bounded absorbing set.
\bp
\label{abo}
There exists $R_0>0$ such that the ball $\mathcal{B}_0$ in $\mathbb{X}_{M, M'}$ of radius $R_0$ centered at zero is absorbing for the semigroup $S(t)$.
Namely, for every bounded set $\mathcal{B}\subset \mathbb{X}_{M, M'}$, there exists $t_0=t_0(\mathcal{B}, M, M')$ such that
$$ S_2(t) \mathcal{B}\subset \mathcal{B}_0, \quad \forall\, t\geq t_0.$$
\ep
\begin{proof}
For  every bounded set $\mathcal{B}\subset \mathbb{X}_{M, M'}$, consider an initial datum
$(\theta_0, \mathbf{q}_0, \chi_0, \xi_0, \chi_1)\in \mathcal{B}\subset \mathbb{X}_{M, M'}$. Then we have
\be
\|(\theta_0, \mathbf{q}_0, \chi_0, \xi_0,\chi_1) \|_\mathbb{X}\leq R,\non
\ee
where $R>0$ is a constant depending on $\mathcal{B}$. Besides, we observe that
\be
|\langle\theta_0\rangle|\leq 2M+M', \quad |\langle\chi_0\rangle|\leq M+M', \quad |\langle\chi_1\rangle|\leq M'.\non
\ee
Thus, from the definition of $\mathcal{Y}$ (cf. \eqref{Ya}) we infer
\be
\mathcal{Y}(0)\leq C(R, M, M').\non
\ee
It follows from \eqref{Y1} and \eqref{Y2} that there exists $t_0=t_0(R,M,M')>0$ such that
\be
\|\theta(t)\|^2+\|\mathbf{q}(t)\|^2+\|\chi(t)\|_{V}^2+\|\xi(t)\|_{V_\Gamma}^2+\|\chi_t(t)\|_{V^*}^2\leq R_0, \quad \forall\, t\geq t_0,\non
\ee
where $R_0$ may depend on $M$ and $M'$ but is independent of $R$ and $t$. The proof is complete.
\end{proof}
Next, we study the precompactness of trajectories in $\mathbb{X}$.
\bp\label{precom}
Suppose that assumptions (H1)--(H3) are satisfied. Let $(\theta, \mathbf{q}, \chi, \xi, \chi_t)$ be the unique energy solution to problem
\eqref{c1}--\eqref{c7} given by Theorem \ref{exe}-(ii), with initial datum $(\theta_0, \mathbf{q}_0, \chi_0, \xi_0, \chi_1) \in \mathbb{X}$.
Then the orbit $$ \bigcup_{t\geq 0} (\theta, \mathbf{q}, \chi, \xi, \chi_t)(t)$$
is precompact in $\mathbb{X}$.
\ep
\begin{proof}
Similar to \cite{GPS2}, from the assumption (H2), Remark \ref{AH}(1) and the Sobolev embedding theorem,  it follows that there exists a sufficient
large constant $\gamma_0>c_0$ such that
\be
 \frac12\|\nabla z\|^2+(\gamma_0-2c_0)\|z\|^2\geq \int_\Omega f'(\zeta)z^2 dx,\quad \forall\, z, \zeta\in V.\label{k1}
\ee
Similarly, we can find a sufficient large constant $\gamma_1>c_1$ such that
\be
\frac12\|\nabla z\|_{H_\Gamma}^2+(\gamma_1-2c_1)\|z\|_{H_\Gamma}^2\geq \int_\Gamma g'(\zeta)z^2 dS,\quad \forall\, z, \zeta\in V_\Gamma.\label{k2}
\ee
We introduce
\be\label{k2bis}
\hat{f}(y)=f(y)+\gamma_0 y, \quad \hat{g}(y)=g(y)+\gamma_1 y, \quad y\in \mathbb{R}.
\ee
It is clear that $\hat{f}, \hat{g}$ are monotone nondecreasing functions in $\mathbb{R}$.
Then we split the solution to problem \eqref{c1}--\eqref{c7} as follows:
\be
(\theta, \mathbf{q}, \chi, \xi, \chi_t)(t)=(\theta^d, \mathbf{q}^d, \chi^d, \xi^d, \chi^d_t)(t)+(\theta^c, \mathbf{q}^c, \chi^c, \xi^c, \chi^c_t)(t),\label{de}
\ee
where
\be
\begin{cases}
& (\theta^d+\chi^d)_t+\nabla \cdot \mathbf{q}^d=0, \quad \text{in \, }  \Omega\times(0,\infty),\\
& \mathbf{q}^d_t+\mathbf{q}^d+\nabla \theta_d=0,\quad \text{in \, }  \Omega\times(0,\infty),\\
& \chi^d_{tt}+\chi^d_t+A\mu^d=0, \quad \text{in \, }  \Omega\times(0,\infty),\\
&\mu^d=A\chi^d+\hat{f}(\chi)-\hat{f}(\chi^c)+\alpha \chi^d_t-\theta^d,\quad \text{in \, }  \Omega\times(0,\infty),\\
& \mathbf{q}^d\cdot \nu =\partial_\nu \mu^d=0, \quad \text{on \, }  \Gamma\times(0,\infty),\\
& \xi^d_t-\Delta_\Gamma \xi^d+\partial_\nu \chi^d+\hat{g}(\xi)-\hat{g}(\xi^c)=0, \quad \text{on \, }  \Gamma\times(0,\infty),\\
& \theta^d(0)=\tilde{\theta_0}, \ \mathbf{q}^d(0)=\mathbf{q}_0,\  \chi^d(0)=\tilde{\chi}_0, \ \xi^d(0)=\xi_0, \ \chi^d_t(0)=\tilde{\chi}_1,
\quad \text{on \, }  \Omega,\end{cases}\label{d}
\ee
and
\be
\begin{cases}
&   (\theta^c+\chi^c)_t+\nabla \cdot \mathbf{q}^c=0, \quad \text{in \, }  \Omega\times(0,\infty),\\
& \mathbf{q}^c_t+\mathbf{q}^c+ \nabla \theta^c=0, \quad \text{in \, }  \Omega\times(0,\infty),\\
& \chi^c_{tt}+\chi^c_t+A\mu^c=0, \quad \text{in \, }  \Omega\times(0,\infty),\\
& \mu^c=A\chi^c+\hat{f}(\chi^c)+\alpha \chi^c_t-\theta^c-\gamma_0 \chi,\quad \text{in \, }  \Omega\times(0,\infty),\\
& \mathbf{q}^c\cdot \nu =\partial_\nu \mu^c=0, \quad \text{on \, }  \Gamma\times(0,\infty),\\
&  \xi^c_t-\Delta_\Gamma \xi^c+\partial_\nu \chi^c+\hat{g}(\xi^c)=\gamma_1 \xi, \quad \text{on \, }  \Gamma\times(0,\infty),\\
& \theta^c(0)=\langle\theta_0\rangle, \ \mathbf{q}^c(0)=\mathbf{0},\  \chi^c(0)=\langle\chi_0\rangle, \ \xi^c(0)=0, \ \chi^c_t(0)=\langle\chi_1\rangle,
\quad \text{on \, }  \Omega.\end{cases}\label{c}
\ee
In \eqref{c}, we consider $(\chi, \xi)\in L^\infty(0, \infty; \mathbb{H}^1)$ as given.
Then, in analogy to the proof of Lemma \ref{es}, we can prove that problem \eqref{c} admits a unique global solution such that
\bea
&&\|(\theta^c, \mathbf{q}^c, \chi^c, \xi^c, \chi^c_t)(t)\|_{\mathbb{X}}\leq C, \quad \forall\, t\geq 0, \label{ec1}\\
&&\int_t^{t+1} (\alpha \|\chi^c_t(\tau)\|^2+\|\xi^c_t(\tau)\|_{H_\Gamma}^2+ \|(\theta^c, \mathbf{q}^c, \chi^c, \xi^c, \chi_t^c)(\tau)\|_\mathbb{X}^2)d\tau \leq C,
\quad \forall\, t\geq 0.\label{ec2}
\eea
Due to \eqref{disa1} and the decomposition \eqref{de}, we obtain similar uniform estimates for the decay part $(\theta^d, \mathbf{q}^d, \chi^d, \xi^d, \chi^d_t)(t)$.

Next, we show that $\|(\theta^d, \mathbf{q}^d, \chi^d, \xi^d, \chi^d_t)(t)\|_\mathbb{X}$ indeed decays to zero exponentially fast as time tends to infinity.
Due to the choice of initial data, it is easy to verify that
\be
\langle \theta^d(t)\rangle=\langle \chi^d(t)\rangle=\langle \chi^d_t(t)\rangle=0, \quad \forall\, t\geq 0.\label{dzeromean}
\ee
As in the proof of Lemma \ref{es}, in \eqref{d} we multiply the first equation by $\theta^d$, the second equation by $\mathbf{q}^d$,
the third equation by $A_0^{-1}(\chi_t^d+\kappa_1\chi^d)$ and we integrate over $\Omega$.
Then, summing up all the resulting equations and adding the functional $\kappa_2(\mathbf{q}^d, \nabla A_0^{-1}\theta^d)$ ($\kappa_1, \kappa_2$ are
positive constants to be determined later), we have
\be
\frac{d}{dt} \mathcal{Y}^d(t) +\mathcal{I}^d(t)\leq \mathcal{R}^d(t), \quad \forall \, t \geq 0,
\label{dd}
\ee
where
\bea
\mathcal{Y}^d&=&\frac12\|\theta^d\|^2+\frac{1}{2}\|\mathbf{q}^d\|^2
+\frac{1}{2}\|A_0^{-\frac12}\chi^d_t\|^2+ \frac12\|\nabla \chi^d\|^2+\int_\Omega (\hat{f}(\chi)-\hat{f}(\chi^c))\chi^d dx\non\\
&&  -\frac12 \int_\Omega \hat{f}'(\chi) (\chi^d)^2dx +\frac12\|\nabla_\Gamma \xi^d\|_{H_\Gamma}^2+\frac{\kappa_1}{2}\|\xi^d\|_{H_\Gamma}^2\non\\
&& + \int_\Gamma (\hat{g}(\xi)-\hat{g}(\xi^c)) \xi^d dS-\frac12 \int_\Gamma \hat{g}'(\xi)(\xi^d)^2 dS
+ \kappa_1\int_\Omega A_0^{-\frac12}\chi^d_t A_0^{-\frac12}\chi^d dx\non\\
&& +\frac{\kappa_1}{2}\| A_0^{-\frac12}\chi^d\|^2+\frac{\kappa_1\alpha}{2}\|\chi^d\|^2+ \kappa_2\int_\Omega \mathbf{q}\cdot\nabla A_0^{-1}\theta^d dx,\non
\eea
\bea
\mathcal{I}^d&=& \|\mathbf{q}^d\|^2+(1-\kappa_1)\|A_0^{-\frac12}\chi^d_t\|^2+\alpha \|\chi^d_t\|^2+\|\xi^d_t\|^2_{H_\Gamma}+\kappa_1\|\nabla \chi^d\|^2\non\\
&& +\kappa_1\|\nabla_\Gamma \xi^d\|^2_{H_\Gamma}
+\kappa_1\left(\int_\Omega (\hat{f}(\chi)-\hat{f}(\chi^c))\chi^d dx+\int_\Gamma (\hat{g}(\xi)-\hat{g}(\xi^c))\xi^d dS\right)\non\\
&& +\kappa_2\|\theta^d\|^2\non
\eea
and
\bea
\mathcal{R}^d & =& \int_\Omega (\hat{f}'(\chi)-\hat{f}'(\chi^c))\chi^c_t \chi^d dx-\int_\Omega \hat{f}''(\chi)\chi_t(\chi^d)^2dx
+ \int_\Gamma (\hat{g}'(\xi)-\hat{g}'(\xi^c))\xi^c_t \xi^d dS\non\\
&& -\int_\Omega \hat{g}''(\xi)\xi_t(\xi^d)^2dS-\kappa_2\int_\Omega \mathbf{q}^d\cdot \nabla A_0^{-1} \theta^d dx
-\kappa_2\int_\Omega \mathbf{q}^d\cdot \nabla A_0^{-1}\chi^d_tdx\non\\
&& +\kappa_2\|A_0^{-\frac12}\nabla \cdot \mathbf{q}^d\|^2.\non
\eea
We note that $\hat{f}$ and $\hat{g}$ are monotone nondecreasing functions. Moreover, if $\gamma_0, \gamma_1$ are sufficiently large, we have
\be
\left(\int_\Omega
(\hat{f}(\chi)-\hat{f}(\chi^c))\chi^d dx+\int_\Gamma (\hat{g}(\xi)-\hat{g}(\xi^c))\xi^d dS\right)\geq \|\chi^d\|^2+\|u^d\|_{H_\Gamma}^2,\label{pos}
\ee
which implies that
\be
\mathcal{I}^d(t)\geq C\|(\theta^d, \mathbf{q}^d, \chi^d, \xi^d, \chi^d_t)\|^2_{\mathbb{X}}.\non
\ee

We now estimate $\mathcal{R}^d(t)$. First, we observe that the last three terms can be evaluated exactly as in \eqref{kk}.
Using the uniform estimate \eqref{disa1}, \eqref{ec1}, \eqref{dzeromean}, the Sobolev embedding inequality and the Poincar\'e inequality, we deduce that, for the case $d=3$, there holds (the case $d=2$ is similar)
\color{black}
\bea
&& \int_\Omega (\hat{f}'(\chi)-\hat{f}'(\chi^c))\chi^c_t \chi^d dx-\int_\Omega \hat{f}''(\chi)\chi_t(\chi^d)^2dx\non\\
&\leq& C(1+\|f''\|_{L^6(\Omega)})(\|\chi^c_t\|+\|\chi_t\|)\|\chi^d\|_{L^6(\Omega)}^2\non\\
&\leq& \frac{\kappa_1}{2}\|\nabla \chi^d\|^2+C(\|\chi^c_t \|^2+\|\chi_t\|^2)\|\nabla \chi^d\|^2.\non
\eea
Similarly, we have
\bea
&& \int_\Gamma (\hat{g}'(\xi)-\hat{g}'(\xi^c))\xi^c_t \xi^d dS-\int_\Omega \hat{g}''(\xi)\xi_t(\xi^d)^2dS\non\\
&\leq &  \frac{\kappa_1}{2}(\|\nabla_\Gamma \xi^d\|_{H_\Gamma}^2+\|\xi^d\|_{H_\Gamma}^2)\non\\
&& +C(\|\xi^c_t \|_{H_\Gamma}^2+\|\xi_t \|_{H_\Gamma}^2)(\|\nabla_\Gamma \xi^d\|_{H_\Gamma}^2+\|\xi^d\|_{H_\Gamma}^2).\non
\eea
Due to \eqref{k1} and \eqref{k2}, we get
\bea
&& \int_\Omega (\hat{f}(\chi)-\hat{f}(\chi^c))\chi^d dx -\frac12 \int_\Omega \hat{f}'(\chi) (\chi^d)^2dx\non\\
&\geq & \frac12 (\gamma_0-2c_0) \|\chi^d\|^2-\frac12 \int_\Omega f'(\chi) (\chi^d)^2dx\non\\
&\geq & -\frac14\|\nabla \chi^d\|^2,\non
\eea
and
\be
\int_\Gamma (\hat{g}(\xi)-\hat{g}(\xi^c)) \xi^d dS-\frac12 \int_\Gamma \hat{g}'(\xi)(\xi^d)^2 dS\geq -\frac14\|\nabla_\Gamma \xi^d\|_{H_\Gamma}^2.\non
\ee
Then, taking $\kappa_1$ and $\kappa_2$ small enough, we can find $\eta>1$ such that
\be
\eta^{-1}\|(\theta^d, \mathbf{q}^d, \chi^d, \xi^d, \chi^d_t)\|^2_{\mathbb{X}}
\leq \mathcal{Y}^d\leq \eta \|(\theta^d, \mathbf{q}^d, \chi^d, \xi^d, \chi^d_t)\|^2_{\mathbb{X}}.\label{equa}
\ee
Thus, from the above estimate and \eqref{dd} we infer that there exist two positive constants $K_1, K_2$ such that the following estimate holds
\be
\frac{d}{dt}\mathcal{Y}^d+K_1\mathcal{Y}^d\leq K_2(\|\chi^c_t \|^2+\|\chi_t\|^2+\|\xi^c_t \|_{H_\Gamma}^2+\|\xi_t \|_{H_\Gamma}^2)\mathcal{Y}^d.\non
\ee
In \eqref{c}, we take $v=\tth^c$ in the first equation, $\mathbf{v}=\mathbf{q}^c$ in the second equation and $w=A^{-1}_0 \tch^c_t$ in the third equation.
Adding the results together, we obtain
\bea
&& \frac{d}{dt}\left(\frac12\|\tth^c\|^2+\frac{1}{2}\|\tq^c\|^2 +\frac{1}{2}\|A_0^{-\frac12}\tch^c_t\|^2
+ \frac12\|\nabla \tch^c\|^2+\int_\Omega \hat{F}(\chi^c)dx \right.\non\\
&& \left. -\gamma_0\int_\Omega \chi\tch^c dx+\frac12\|\nabla_\Gamma \txi^c\|_{H_\Gamma}^2
+ \int_\Gamma \hat{G}(\xi^c)dS -\gamma_1\int_\Gamma \xi\txi^c dS\right)\non\\
&&  +\|\tq^c\|^2+\|A_0^{-\frac12}\tch^c_t\|^2+\alpha \|\tch^c_t\|^2+\|\txi^c_t\|^2_{H_\Gamma}\non\\
& =& \langle\chi^c_t\rangle\int_\Omega \hat{f}(\chi^c) dx -\gamma_0\int_\Omega \chi_t\tch^c dx
+\langle\chi^c_t\rangle\int_\Gamma \hat{g}(\xi^c) dS-\gamma_1\int_\Gamma u_t \txi^cdS\non\\
&&-\langle\chi^c_t\rangle\int_\Gamma \txi^c_t dS.\label{k3}
\eea
On account of the fact $\langle \chi^c_t(t)\rangle=\langle\chi_t(t)\rangle=\langle\chi_1\rangle e^{-t}$, for all $t\geq 0$,
and the uniform estimates \eqref{disa1}, \eqref{ec2},
then, $\forall \, \epsilon>0$, the right-hand side \eqref{k3} can be evaluated as follows
\bea
& \langle\chi^c_t(t)\rangle\int_\Omega \hat{f}(\chi^c)(t) dx -\gamma_0\int_\Omega \chi_t(t)\tch^c (t)dx
+ \langle\chi^c_t(t)\rangle\int_\Gamma \hat{g}(\xi^c)(t) dS\non\\
&-\gamma_1\int_\Gamma \xi_t (t)\txi^c(t)dS-\langle\chi^c_t(t)\rangle\int_\Gamma \txi^c_t (t)dS\non\\
&\leq |\langle\chi_1\rangle|e^{-t}(\|\hat{f}(\chi^c)(t)\|_{L^1(\Omega)}+\|\hat{g}(\xi^c)(t)\|_{L^1(\Gamma)}+\|\txi^c_t(t)\|_{L^1(\Gamma)})\non\\
&+C\|\chi_t(t)\|\|\tch^c(t)\|+C\|\xi_t(t)\|_{H_\Gamma}\| \txi^c(t)\|_{H_\Gamma}\non\\
&\leq \frac12 \|\txi^c_t(t)\|^2_{H_\Gamma}+ \epsilon+\frac{C}{\epsilon}(\|\chi_t(t)\|^2+\|\xi_t(t)\|^2_{H_\Gamma})
+C(e^{-t}+e^{-2t}), \quad \forall\, t \geq 0.\label{k4}
\eea
As a result, for any $\epsilon>0$, we deduce from \eqref{k3}, \eqref{k4} and \eqref{kka} that
\be
\int_s^t \alpha \|\tch^c_t(\tau)\|^2+\|\txi^c_t(\tau)\|^2_{H_\Gamma} d\tau\leq \epsilon(t-s)+ C_\epsilon, \quad \forall \, t>s>0,\non
\ee
from which, combining with  estimate \eqref{kka},  we infer, for any $\epsilon>0$,
\bea
&& \int_s^t  K_2(\|\chi^c_t \|^2+\|\chi_t\|^2+\|\xi^c_t \|_{H_\Gamma}^2+\|\xi_t \|_{H_\Gamma}^2)d\tau\non\\
&\leq& \epsilon(t-s)+ C_\epsilon, \quad \forall\, t>s>0.\label{sss}
\eea
Then, an application of the Gronwall-type lemma (see e.g., \cite[Lemma 2.2]{GP}) allows to conclude that
\be
\mathcal{Y}^d(t)\leq C\mathcal{Y}(0)e^{-\frac{K_1}{2}t}, \quad \forall\, t\geq 0,\label{expatt}
\ee
which, together with \eqref{equa}, yields the exponential decay of $(\theta^d, \mathbf{q}^d, \chi^d, \xi^d, \chi^d_t)(t)$ in $\mathbb{X}$.

Finally, we prove that $(\theta^c, \mathbf{q}^c, \chi^c, \xi^c, \chi^c_t)(t)$ is bounded in a space that can be compactly embedded into $\mathbb{X}$.
To this aim, we (formally) differentiate \eqref{c} with respect to time to get
\be
\begin{cases}
&  (\theta^c_t+\chi^c_t)_{t}+\nabla \cdot \mathbf{q}^c_t=0, \quad \text{in \, }  \Omega\times(0,\infty),\\
& \mathbf{q}^c_{tt}+\mathbf{q}^c_t+ \nabla \theta^c_t=0, \quad \text{in \, }  \Omega\times(0,\infty),\\
& \chi^c_{ttt}+\chi^c_{tt}+A\mu^c_t=0, \quad \text{in \, }  \Omega\times(0,\infty),\\
& \mu^c_t=A\chi^c_t+\hat{f}'(\chi^c)\chi^c_t+\alpha \chi^c_{tt}-\theta^c_t-\gamma_0 \chi_t,\quad \text{in \, }  \Omega\times(0,\infty),\\
& \mathbf{q}^c_t\cdot \nu =\partial_\nu \mu^c_t=0, \quad \text{on \, }  \Gamma\times(0,\infty),\\
&  \xi^c_{tt}-\Delta_\Gamma \xi^c_t+\partial_\nu \chi^c_t+\hat{g}'(\xi^c)\xi^c_t=\gamma_1 \xi_t, \quad \text{on \, }  \Gamma\times(0,\infty),\\
&\theta^c_t(0)=0, \ \mathbf{q}^c_t(0)=\mathbf{0},\  \chi^c_t(0)=\langle\chi_1\rangle, \quad \text{in \, }  \Omega, \\
& \xi^c_t(0)=-\hat{g}(0)+\gamma_1 u_0,
\ \chi^c_{tt}(0)=\langle\chi_1\rangle+\gamma_0A\chi_0, \quad \text{in \, }  \Omega.
\end{cases}\label{cc}
\ee
We recall that
\be
\langle\chi_t^c(t)\rangle=\langle\chi_1\rangle e^{-t}=Q_1(t), \quad
\langle\theta^c_t(t)\rangle=\langle\chi_{tt}^c(t)\rangle=-\langle\chi_t^c(t)\rangle= -Q_1(t), \quad t \geq 0. \label{cond}
\ee
In \eqref{cc}, we multiply the first equation by $\tth^c_t$ and  the second equation by $ \mathbf{q}_t^c$. Adding the resulting equations
together, we get
\be
\frac12\frac{d}{dt}(\| \tth^c_t\|^2+\|\mathbf{q}^c_t\|^2)+\|\mathbf{q}_t^c\|^2=-\int_\Omega \tch^c_{tt} \tth^c_t dx.\label{cca}
\ee
Multiplying the third equation of \eqref{c} by $A_0^{-1}\tch_{tt}^c$ an integrating over $\Omega$, we obtain
\bea
&&\frac{d}{dt}\left(\frac12\|A_0^{-\frac12}\tch_{tt}^c\|^2+\frac12\|\nabla\tch_t^c\|^2 +\frac12\|\nabla_\Gamma\txi^c_t\|_{H_\Gamma}^2
+\frac12\int_\Omega \hat{f}'(\chi^c)(\tch_t^c)^2 dx +\frac12\int_\Gamma \hat{g}'(\xi^c)(\txi_t^c)^2 dS\right)\non\\
&&  +\|A_0^{-\frac12}\tch_{tt}^c\|^2+\alpha \|\tch_{tt}^c\|^2+\|\txi_{tt}^c\|_{H_\Gamma}^2\non\\
&=&\frac12\int_\Omega \hat{f}''(\chi^c)\chi^c_t(\tch^c_{t})^2dx+\frac12\int_\Gamma \hat{g}''(\xi^c)\xi^c_t(\txi^c_{t})^2dS
+\gamma_0\int_\Omega \tch_t \tch^c_{tt} dx+ \gamma_1\int_\Gamma \xi_t \txi^c_{tt}dS\non\\
&& -Q_1\int_\Omega \hat{f}'(\chi^c)\tch^c_{tt}dx-Q_1\int_\Gamma \hat{g}'(\xi^c)\txi^c_{tt}dS+\int_\Omega \tch^c_{tt} \tth^c_t dx
+ Q_1\int_\Gamma \txi_{tt}^c dS.\label{ccb}
\eea
On the other hand, multiplying the third equation of \eqref{c} by $A_0^{-1}\tch_{t}^c$, after an integration by parts we get
\bea
&&\frac{d}{dt}\left( \int_\Omega A_0^{-\frac12} \tch_{tt}^cA_0^{-\frac12}\tch_t^cdx +\frac12\|A_0^{-\frac12}\tch^c_{t}\|^2
+\frac{\alpha}{2}\|\tch^c_t\|^2+\frac12\|\txi^c_t\|^2_{H_\Gamma}\right)\non\\
&& -\|A_0^{-\frac12}\tch_{tt}^c\|^2+ \|\nabla \tch_t^c\|^2 +\int_\Omega \hat{f}'(\chi^c)(\tch^c_t)^2 dx
+\|\nabla _\Gamma \txi^c_t\|^2_{H_\Gamma}+\int_\Gamma \hat{g}'(\xi^c)(\xi^c_t)^2 dS\non\\
&=& Q_1\int_\Omega \hat{f}'(\chi^c)\tch^c_t dx+\int_\Omega \tth^c_t\tch^c_t dx+\gamma_0\int_\Omega \tch_t\tch^c_t dx\non\\
&& +Q_1\int_\Gamma \hat{g}(\xi^c)\txi^c_t dS+Q_1\int_\Gamma \txi ^c_t dS+\gamma_1\int_\Gamma \xi_t \txi^c_t dS. \label{ccc}
\eea
On the other hand, we have
\bea
&& \frac{d}{dt}\int_\Omega \mathbf{q}^c_t\cdot \nabla A_0^{-1}\tth^c_t dx+\|\tth^c_t\|^2
\non\\
&=&\int_\Omega \mathbf{q}^c_{tt} \cdot  \nabla A_0^{-1}\tth^c_t dx+ \int_\Omega \mathbf{q}^c_t \cdot   \nabla A_0^{-1}\tth^c_{tt} dx+\|\tth^c_t\|^2\non\\
&=& -\int_\Omega \mathbf{q}^c_t\cdot \nabla A_0^{-1} \tth^c_t dx -\int_\Omega \mathbf{q}^c_t\cdot \nabla A_0^{-1}\tch^c_{tt}dx
+\|A_0^{-\frac12}\nabla \cdot \mathbf{q}^c_t\|^2.
\label{ccd}
\eea
Multiplying \eqref{ccc} by $\kappa_1>0$ and \eqref{ccd} by $\kappa_2>0$, respectively, and adding the results with \eqref{cca} and \eqref{ccb},
then we obtain, for any $t\geq 0$,
\be
\frac{d}{dt}\mathcal{Y}^c(t)+ \mathcal{I}^c(t)\leq \mathcal{R}^c(t),\label{yc}
\ee
where
\bea
\mathcal{Y}^c&=&\frac12(\| \tth^c_t\|^2+\|\mathbf{q}^c_t\|^2)+\frac12\|A_0^{-\frac12}\tch_{tt}^c\|^2+\frac12\|\nabla\tch_t^c\|^2
+\frac12\|\nabla_\Gamma\txi^c_t\|_{H_\Gamma}^2\non\\
&& +\frac12\int_\Omega \hat{f}'(\chi^c)(\tch_t^c)^2 dx +\frac12\int_\Gamma \hat{g}'(\xi^c)(\txi_t^c)^2 dS
+\kappa_2\int_\Omega \mathbf{q}^c_t\cdot \nabla A_0^{-1}\tth^c_t dx\non\\
&& +\kappa_1\int_\Omega A_0^{-\frac12} \tch_{tt}^cA_0^{-\frac12}\tch_t^cdx +\frac{\kappa_1}{2}\|A_0^{-\frac12}\tch^c_{t}\|^2
+\frac{\kappa_1\alpha}{2}\|\tch^c_t\|^2+\frac{\kappa_1}{2}\|\txi^c_t\|^2_{H_\Gamma},\non
\eea
\bea
\mathcal{I}^c&=& \|\mathbf{q}_t^c\|^2+ (1-\kappa_1) \|A_0^{-\frac12}\tch_{tt}^c\|^2+\alpha \|\tch_{tt}^c\|^2+\|\txi_{tt}^c\|_{H_\Gamma}^2
+ \kappa_1\|\nabla \tch_t^c\|^2\non\\
&& +\kappa_1\int_\Omega \hat{f}'(\chi^c)(\tch^c_t)^2 dx +\kappa_1\|\nabla _\Gamma \txi^c_t\|^2_{H_\Gamma}
+\kappa_1\int_\Gamma \hat{g}'(\xi^c)(\xi^c_t)^2 dS+\kappa_2\|\tth^c_t\|^2\non\\
&& +\kappa_2\int_\Omega \mathbf{q}^c_t\cdot \nabla A_0^{-1} \tth^c_t dx +\kappa_2\int_\Omega \mathbf{q}^c_t\cdot \nabla A_0^{-1}\tch^c_{tt}dx
-\kappa_2\|A_0^{-\frac12}\nabla \cdot \mathbf{q}^c_t\|^2,\non
\eea
and
\bea
\mathcal{R}^c&=&\frac12\int_\Omega \hat{f}''(\chi^c)\chi^c_t(\tch^c_{t})^2dx+\frac12\int_\Gamma \hat{g}''(\xi^c)\xi^c_t(\txi^c_{t})^2dS
+\gamma_0\int_\Omega \tch_t \tch^c_{tt} dx\non\\
&& + \gamma_1\int_\Gamma \xi_t \txi^c_{tt}dS-Q_1\int_\Omega \hat{f}'(\chi^c)\tch^c_{tt}dx-Q_1\int_\Gamma \hat{g}'(\xi^c)\txi^c_{tt}dS
+ Q_1\int_\Gamma \txi_{tt}^c dS\non\\
&&+\kappa_1 Q_1\int_\Omega \hat{f}'(\chi^c)\tch^c_t dx+\kappa_1\int_\Omega \tth^c_t\tch^c_t dx+\kappa_1\gamma_0\int_\Omega \tch_t\tch^c_t dx\non\\
&& +\kappa_1Q_1\int_\Gamma \hat{g}(\xi^c)\txi^c_t dS+\kappa_1Q_1\int_\Gamma \txi ^c_t dS+\kappa_1\gamma_1\int_\Gamma \xi_t \txi^c_t dS.\non
\eea
Since $\hat{f}',\, \hat{g}' \geq 1$, taking $\kappa_1, \,\kappa_2$ small enough, we still have
\be
\eta_1^{-1}\|(\tth_t^c, \mathbf{q}^c_t, \tch^c_t, \txi^c_t, \tch^c_{tt})\|^2_{\mathbb{X}}\leq \mathcal{Y}^c
\leq \eta_1\|(\tth_t^c, \mathbf{q}^c_t, \tch^c_t, \txi^c_t, \tch^c_{tt})\|^2_{\mathbb{X}},\label{equb}
\ee
and
\be
\mathcal{I}^c\geq C_1(\|(\tth_t^c, \mathbf{q}^c_t, \tch^c_t, \txi^c_t, \tch^c_{tt})\|^2_{\mathbb{X}})+\frac{\alpha}{2} \|\tch_{tt}^c\|^2
+\frac12\|\txi_{tt}^c\|_{H_\Gamma}^2.\non
\ee
Let us proceed to estimate $\mathcal{R}^c(t)$. Using the uniform bounds \eqref{disa1}, \eqref{ec1} and \eqref{ec2},
on account of H\"older's inequality and Young's inequality, we infer
\bea
\mathcal{R}^c &\leq& C\|\hat{f}''(\chi^c)\|_{L^6(\Omega)}\|\chi^c_t\|\|\tch^c_t\|_{L^6(\Omega)}^2
+C\|\hat{g}''(\xi^c)\|_{L^6(\Gamma)}\|\xi^c_t\|_{H_\Gamma}\|\txi^c_t\|^2_{L^6(\Gamma)} \non\\
&& +\gamma_0\|\tch_{tt}^c\|\|\tch_t\| +\gamma_1\|\txi_{tt}^c\|_{H_\Gamma}\|\xi_t\|_{H_\Gamma}\non\\
&&+Ce^{-t}\| \hat{f}'(\chi^c)\|\|\tch^c_{tt}\|
+Ce^{-t}\| \hat{g}'(\xi^c)\|_{H_\Gamma}\|\txi^c_{tt}\|_{H_\Gamma}  +Ce^{-t}\|\txi^c_{tt}\|_{H_\Gamma}\non\\
&&  + C\kappa_1 e^{-t}\|\hat{f}'(\chi^c)\|\|\tch^c_t\| +\kappa_1\|\tth^c_t\|\|\tch^c_t\|+\kappa_1\gamma_0\| \tch_t\|\|\tch^c_t\| \non\\
 && +C\kappa_1 e^{-t}\| \hat{g}(\xi^c)\|_{H_\Gamma}\|\txi^c_t\| _{H_\Gamma}+ C\kappa_1e^{-t}\| \txi^c_t\|_{H_\Gamma}
 +\kappa_1\gamma_1\| \xi_t \|_{H_\Gamma}\|\txi^c_t\|_{H_\Gamma},\non
\eea
and then, for any $t \geq 0$,
\bea
\mathcal{R}^c(t) &\leq& \frac{\alpha}{4} \|\tch_{tt}^c(t)\|^2+\frac14\|\txi_{tt}^c(t)\|_{H_\Gamma}^2+\frac{\kappa_2}{2}\|\tth_t^c(t)\|^2 \non\\
 && +(\lambda+C\lambda^{-1}\|\chi^c_t(t)\|^2)\|\tch^c_t(t)\|_{V}^2+(\lambda+C\lambda^{-1}\|\xi^c_t(t)\|^2_{H_\Gamma})\|\txi^c_t(t)\|^2_{V_\Gamma}\non\\
 && + C\|\tch_t(t)\|^2+C\|\tch^c_t(t)\|^2+C\|\xi_t(t)\|_{H_\Gamma}^2+C\|\txi^c_{t}(t)\|_{H_\Gamma}^2+Ce^{-2t}.\non
 \eea
\color{black}
Taking $\lambda>0$ small enough, from \eqref{yc} we infer, for any $t \geq 0$,
\bea
\frac{d}{dt}\mathcal{Y}^c(t)+K_1\mathcal{Y}^c(t)
&\leq&  K_2(\|\tch^c_t(t)\|^2+\|\txi^c_{t}(t)\|_{H_\Gamma}^2)\mathcal{Y}^c(t)+Ce^{-2t}\non\\
&&  + \, C\|\tch_t(t)\|^2+C\|\tch^c_t(t)\|^2+C\|\xi_t(t)\|_{H_\Gamma}^2+C\|\txi^c_{t}(t)\|_{H_\Gamma}^2.\non
\eea
Recalling \eqref{sss}, we can apply the Gronwall-type lemma (see e.g., \cite[Lemma 2.2]{GP}) again to conclude that
\be
\mathcal{Y}^c(t)\leq C\mathcal{Y}^c(0)e^{-\frac{K_1}{2}t}+C, \quad \forall\, t\geq 0,\non
\ee
which, together with \eqref{cond} and \eqref{equb}, yields the uniform estimate of $(\theta^c_t, \mathbf{q}^c_t, \chi^c_t, \xi^c_t, \chi^c_{tt})(t)$ in $\mathbb{X}$.
Besides, on account of  \eqref{ec1} we know that $(\theta^c, \mathbf{q}^c, \chi^c, \xi^c, \chi^c_{t})(t)$ is also
uniformly bounded in $\mathbb{X}$. We now use the same argument as in Lemma \ref{esY} to get higher-order estimate.
From equation \eqref{c} we deduce
\be
\|\nabla \theta^c\|\leq \|\mathbf{q}_t^c\|+\|\mathbf{q}^c\|\leq C, \quad \|\nabla \cdot \mathbf{q}^c\|\leq \|\theta^c_t\|+\|\chi^c_t\|\leq C.\non
\ee
Since
$$ (\nabla \times\mathbf{q}^c)_t(t)+(\nabla \times\mathbf{q}^c)(t)=0 \quad \text{for} \ t>0, \quad  \quad \nabla \times\mathbf{q}^c(0)=\mathbf{0},$$
we have $(\nabla \times\mathbf{q}^c)(t)=\mathbf{0}$ for all $t\geq 0$. Thus, $\|\mathbf{q}^c\|_{\mathbf{H}^1(\Omega)}\leq C$.
Next, we rewrite \eqref{c}(4) and \eqref{c}(6) as follows
\bea
& -\Delta \chi^c=\mu^c-\hat{f}(\chi^c)-\alpha \chi^c_t+\theta^c+\gamma_0 \chi:=h_1,\label{ep1}\\
&  -\Delta_\Gamma \xi^c+\partial_\nu \chi^c+\beta \xi^c =- \xi^c_t - \hat{g}(\xi^c)+ \beta \xi^c+ \gamma_1 \xi:=h_2,\label{ep2}
\eea
where $\beta>0$ is a positive constant.
Since $\mu^c$ satisfies
\be
-\Delta \mu^c= -(\chi^c_{tt}+\chi^c_t) \quad \text {in} \ \Omega\times (0,\infty), \ \quad \partial_\nu\mu^c=0 \quad \text{on} \ \Gamma\times (0,\infty),\non
\ee
we see that
\bea
\|\mu^c\|_{V}&\leq& C(\|\chi^c_{tt}+\chi^c_t\|_{V^*}+|\langle\mu^c\rangle|\non\\
&\leq& C+| \langle\hat{f}(\chi^c)\rangle|+\alpha |\langle \chi^c_t\rangle|+|\langle \theta^c \rangle|+\gamma_0 |\langle \chi \rangle|  \non\\
&& + \frac{1}{|\Omega|} \left(   \left|\int_\Gamma \xi^c_t dS\right|+ \left|\int_\Gamma \hat{g}(\xi^c)dS\right|
+ \gamma_1 \left|\int_\Omega \xi dx\right|\right)\non\\
&\leq& C.\non
\eea
Using estimate \eqref{ec1} and the same argument as in Lemma \ref{esY}, we get
\be
\|(\chi^c(t), \xi^c(t))\|_{\mathbb{H}^3}\leq C, \quad \forall\, t\geq 0.\label{highatt}
\ee
Collecting the estimates above, we see that $(\theta^c, \mathbf{q}^c, \chi^c, \xi^c, \chi^c_{t})(t)$ is uniformly bounded in $\mathbb{Y}$,
which is compactly embedded into $ \mathbb{X}$.

In summary, we have proved that any trajectory starting from $\mathbb{X}$ can be decomposed into two parts: one part decays exponentially fast to zero
in $\mathbb{X}$ and the other part is uniformly bounded in $\mathbb{Y}$. Thus, the trajectory is precompact in $\mathbb{X}$. The proof is complete.
\end{proof}

\textbf{Proof of Theorem \ref{GA}}.
Proposition \ref{abo} implies that the semigroup $S_2(t)$ has a bounded absorbing set in $\mathbb{X}_{M,M'}$.
On the other hand, Proposition \ref{precom} yields the precompactness of the trajectory and, in particular, the existence of a compact (exponentially)
attracting set (cf. \eqref{expatt} and \eqref{highatt}). Then the conclusion of Theorem \ref{GA} follows from a classical result in the general theory of infinite dimensional dynamical systems (see, e.g., \cite[Ch.2, Theorem 2.2]{BV} or \cite[Theorem I.1.1]{Temam}).

\br
\label{smoothattr}
Let us consider the \emph{closed semigroup} $S_1(t)$ associated with weak solutions. The existence of the global attractor can be established within the framework introduced in  \cite[Theorem 2]{PZ07} by proving first the existence of an absorbing set in the phase space
$$\mathbb{Y}_{M, M'}=\{(z_1, \mathbf{z}_2, z_3,z_4,z_5)\in \mathbb{Y}:\ |\langle z_1+z_3\rangle|\leq M,\  |\langle z_3+z_5\rangle|\leq M, \ |z_5|\leq M'\}.$$
Then, thanks to decomposition \eqref{de}, one can construct a positively invariant exponential attracting set $\mathcal{B}$ which is bounded in $\mathbb{Y}_{M, M'}$. Using the same decomposition and taking the initial data in $\mathcal{B}$, it is possible to prove the asymptotic compactness of the semigroup in $\mathbb{Y}$. The global attractor coincides with the previous one, that is, we have a smoothness result for $\mathcal{A}$. The details are left to the interested reader.
\er

\section{Convergence to equilibrium}
\label{sec5}
\setcounter{equation}{0}
In this section, we proceed to investigate the long-time behavior of single weak solution for any given initial datum
$(\theta_0, \mathbf{q}_0, \chi_0, \xi_0, \chi_1)\in \mathbb{Y}$.

\subsection{Stationary problem and {\L}ojasiewicz--Simon inequality}

First, we look at the corresponding stationary problem. The steady states $(\theta_\infty, \chi_\infty, \xi_\infty)$
of problem \eqref{c1}--\eqref{c7} satisfy the following elliptic boundary value problem
\be
\begin{cases}
& -\Delta \theta_\infty=0,\quad \text{in \,} \Omega, \non\\
& -\Delta(-\Delta \chi_\infty+f(\chi_\infty)-\theta_\infty)=0,\quad \text{in \,} \Omega,\non\\
& \partial_\nu \theta_\infty=0,\quad x\in \Gamma, \quad \text{in \,} \Omega,\non\\
& \partial_\nu (-\Delta \chi_\infty+f(\chi_\infty))=0, \quad \text{on \,} \Gamma, \non\\
&  -\Delta_\Gamma \xi_\infty +\partial_\nu \chi_\infty +g(\xi_\infty)=0,\quad \text{on \,} \Gamma, \non\\
& \xi_\infty=\chi_\infty|_\Gamma,\non
\end{cases}
\ee
with constraints dictated by the initial data on account of the boundary conditions
$$
\langle\chi_\infty\rangle=\langle \chi_0+\chi_1\rangle,\quad  \langle\theta_\infty\rangle=\langle\theta_0\rangle-\langle\chi_1\rangle.
$$
It is easy to see that the above system can be reduced to the following form:
\be
 \begin{cases}
 & \theta_\infty=\langle\theta_0\rangle-\langle\chi_1\rangle,\\
 & -\Delta \chi_\infty+f(\chi_\infty)=\mu_\infty,\quad \text{in \,} \Omega,\\
 & -\Delta_\Gamma \xi_\infty +\partial_\nu \chi_\infty +g(\xi_\infty)=0,\quad \text{on \,} \Gamma, \\
 & \xi_\infty=\chi_\infty|_\Gamma,\\
 & \langle\chi_\infty\rangle=\langle \chi_0+\chi_1\rangle,\\
 \end{cases}
  \label{sta}
 \ee
 where $\mu_\infty $ is a constant uniquely determined by
 \be
 \mu_\infty= \langle f(\chi_\infty)\rangle+\frac{1}{|\Omega|}\int_\Gamma g(\xi_\infty)dS.\label{muinf}
 \ee
We introduce  the functional
\be
\Upsilon(u, v)=\frac12\|\nabla u\|^2 + \frac12\| \nabla_\Gamma v\|_{H_\Gamma}^2+\int_\Omega \widehat{F}(u) dx+\int_\Gamma \widehat{G}(v)dS,\label{upsilon}
\ee
for any $(u, v)\in \mathbb{H}_0^1$ (see Section 2), where
\be
\widehat{F}(u)=F(u+\langle \chi_0+\chi_1\rangle), \quad \widehat{G}(v)=G(v+\langle \chi_0+\chi_1\rangle).
\ee
For any $(u,v), (w,w_\Gamma)\in \mathbb{H}_0^1$, we define the operator
\bea
&& (\mathcal{M}(u,v), (w,w_\Gamma))_{(\mathbb{H}_0^1)^*, \mathbb{H}_0^1 }\non\\
&:=& ( \partial \Upsilon(u, v), (w,w_\Gamma))_{(\mathbb{H}_0^1)^*, \mathbb{H}_0^1} \non\\
&=& \int_\Omega (\nabla u\cdot \nabla w + \widehat{f}(u)w )dx
+ \int_\Gamma (\nabla_\Gamma v\cdot \nabla_\Gamma w_\Gamma + \widehat{g}(v) w_\Gamma) dS.\label{MM1}
\eea
If we restrict the operator $\mathcal{M}$ on $\mathbb{H}_0^2$, i.e., for $(u,v)\in
\mathbb{H}_0^2$,  after integration by parts, from \eqref{MM1} we infer
\be
\mathcal{M}(u,v)  =\mathbb{A} + \left(
\begin{array}{ll}   {\rm P_0} \widehat{f}(u)&0 \\
0& \widehat{g}(v)\\
\end{array}
\right).
\label{cwa}
\ee
Here, we denote by ${\rm P_0}$ the projection operator ${\rm P_0}: H\to H_0$ such that ${\rm P_0}u=u-\langle u\rangle$ for any $u\in H$.
The operator $\mathbb{A}$ is given by
\be
\mathbb{A}  =  \left(
\begin{array}{ll}  {\rm P_0}(-\Delta) & 0 \\
\partial_\nu & -\Delta_{\Gamma} \\
\end{array}
\right).
\label{oA}
\ee
From the identities \eqref{MM1}--\eqref{oA} it easily follows
\bp
Suppose that $(\chi, \xi):=(u+\langle \chi_0+\chi_1\rangle,v+\langle \chi_0+\chi_1\rangle)\in \mathbb{H}^1$ with $\langle u\rangle=0$
is a weak solution to problem \eqref{sta}. Then $(u, v)$ is a critical point of the functional $\Upsilon \in \mathbb{H}_0^1$.
Conversely, if $(u, v)$ is a critical point of the functional $\Upsilon \in \mathbb{H}_0^1$, then
$(\chi, \xi):=(u+\langle \chi_0+\chi_1\rangle,v+\langle \chi_0+\chi_1\rangle)$ is a weak solution to problem \eqref{sta}.
\ep
Furthermore, applying the method of minimizing sequence similar to the one used in \cite{WZ04}, we easily prove the following
\bp
Under assumptions (H1)--(H3), the stationary problem \eqref{sta} admits at least one solution $(\chi_\infty, \xi_\infty)\in \mathbb{H}^1$
and $\theta_\infty$ is given by $\theta_\infty=\langle\theta_0\rangle-\langle\chi_1\rangle$ such that
$$
\Upsilon(\chi_\infty-\langle \chi_0+\chi_1\rangle, \xi_\infty-\langle \chi_0+\chi_1\rangle)=\inf_{(u,v)\in \mathbb{H}_0^1} \Upsilon(u,v).
$$
\ep
\br
By the elliptic estimate (cf. e.g., \cite[Lemma A.1, Corollary A.1]{MZ05}), if $(\chi, \xi)\in \mathbb{H}^1$ is a weak solution to problem \eqref{sta},
then $(\chi, \xi)\in \mathbb{H}^s$ $(s\in \mathbb{N})$, provided that $f, g$ are smooth enough.
\er
Next, we introduce a \L ojasiewicz--Simon type inequality which will be used to prove long-time behavior of global solutions to problem \eqref{c1}--\eqref{c7}.
\bl \label{LS}
Assume that $f, g$ are real analytic and (H2), (H3) are satisfied. Let $(u_*, v_*)\in \mathbb{H}_0^2$ be a critical point of the functional $\Upsilon$.
Then there exist two constants $\rho\in (0, \frac12)$ and $\beta>0$, depending on $(u_*, v_*)$, such that, for any $(u, v)\in \mathbb{H}^1_0$ with
$\|(u, v)-(u_*, v_*)\|_{\mathbb{V}_1}<\beta$, we have
\be
\|\mathcal{M}(u,v)\|_{(\mathbb{H}_0^1)^*}\geq |\Upsilon(u,v)-\Upsilon(u_*, v_*)|^{1-\rho}.\label{ls1}
\ee
\el
\begin{proof}
The proof follows from an argument similar to the one used in \cite{SW}. Here, we just point out some differences.
By the assumptions, $\Upsilon$ is twice Fr\'echet differentiable with respect to the topology of $\mathbb{H}^2$.
Moreover, by the Sobolev embedding $H^2(\Omega)\hookrightarrow L^\infty(\Omega)$ $(n\leq 3)$, $\Upsilon$ is real analytic.
As in \cite{RZ03} and using the Poincar\'e inequality, we can easily show that $\mathbb{A}$ is a strictly positive
self-adjoint unbounded operator from
$D(\mathbb{A})=\{(u, v)\in \mathbb{H}_0^1: \ \mathbb{A}(u, v)\in \mathbb{H}_0\}$ into $\mathbb{H}_0$.
Standard spectral theory allows us to define the power $\mathbb{A}^s$ ($s\in \mathbb{R}$), and we infer that
there exists a complete orthonormal family $\{(\phi_j, \psi_j)\} \in D(\mathbb{A})$, ($j\in \mathbb{N}$, $s\in
\mathbb{R}$), as well as a sequence of eigenvalues $0<\lambda_1\leq \lambda_2\leq...$, $\lambda_j\rightarrow \infty$
as $j$ tends to infinity, such that
\be
\mathbb{A}(\phi_j, \psi_j)^T=\lambda_j(\phi_j, \psi_j)^T,\quad j\in \mathbb{N}.
\ee
In particular, $D(\mathbb{A}^\frac12)=\mathbb{H}_0^1$, $D(\mathbb{A})= \mathbb{H}_0^2$.
By a bootstrap argument, we get $(\phi_j, \psi_j)\in C^\infty(\Omega) $, for all $j\in \mathbb{N}$. For any $(u, v)\in \mathbb{H}_0^1$, we have
\be (\mathbb{A}(u, v)^T, (u, v)^T)_{(\mathbb{H}_0^1)^*,\mathbb{H}_0^1}=\int_\Omega |\nabla u|^2 dx+\int_\Gamma |\nabla_\Gamma v|^2 dS
=\|(u, v)\|^2_{\mathbb{H}_0^1}.\label{1q}
\ee
Following the idea used in \cite{HJ99}, we now introduce the orthogonal projector $P_m$ in $\mathbf{V}_0$ onto
$K_m:=span\{(\phi_1, \psi_1)...,(\phi_m, \psi_m)\}\subset C^\infty(\Omega)$. As in \cite{SW}, we have, for any $(u,v)\in \mathbb{H}_0^1$,
\bea (\mathbb{A}(u,v)^T+\lambda_m P_m (u,v)^T, (u,v)^T)_{(\mathbb{H}_0^1)^*,\mathbb{H}_0^1}
&\geq& \frac12 \|(u, v)\|^2_{\mathbb{H}_0^1}+ \frac14 \lambda_m
\|(u, v)\|_{\mathbb{H}_0}^2.\label{AC}
\eea
Next, we consider the following linearized operator on $\mathbb{H}_0^2$
\be
L(u,v) :=\partial \mathcal{M}(u,v)=\mathbb{A} + \left(
\begin{array}{ll}   {\rm P_0} \widehat{f}'(u)&0 \\
0& \widehat{g}'(v)\\
\end{array}
\right).
\label{cw3.1}
\ee
In analogy to \cite[Lemma 2.3]{WZ04}, we can easily show that $L(u,v)$ is self-adjoint on $\mathbb{H}_0$.
We associate with the operator $L(u_*,v_*)$ the following bilinear form
$b((w_1, w_{1\Gamma}), (w_2, w_{2\Gamma}))$ on $\mathbb{H}_0^1$, for any $(w_1, w_{1\Gamma}),(w_2, w_{2\Gamma})\in \mathbb{H}_0^1$,
\bea
&& b((w_1, w_{1\Gamma}), (w_2, w_{2\Gamma})) \non\\
&= & \int_\Omega (\nabla w_1 \cdot \nabla w_2 + \widehat{f}'(u_*)w_1w_2 )dx
+\int_\Gamma \left( \nabla_\Gamma w_{1\Gamma} \cdot \nabla_\Gamma w_{2\Gamma} +  \widehat{g}'(v_*)w_{1\Gamma} w_{2\Gamma}\right)dS.
\label{cw2.8}
\eea
Since $(u_*, v_*)\in \mathbb{V}_2$, then, by the Sobolev embedding theorems, we infer that
 $L(u_*, v_*)+\lambda_m P_m$ is coercive in $\mathbb{H}_0^1$, provided that $\lambda_m$ is sufficiently large, e.g.,
\be
\lambda_m> 4\max\{\|\widehat{f}'(u_*)\|_{L^\infty(\Omega)},
\|\widehat{g}'(v_*)\|_{L^\infty(\Gamma)}\}.\non
\ee
After establishing the above framework, the proof of the extended \L ojasiewicz--Simon inequality \eqref{ls1} can be reproduced taking advantage of
the arguments used in \cite{HJ99} (see also \cite[Theorem 3.1]{SW}) with minor modifications. The details are omitted here.
\end{proof}

\subsection{Convergence to a single equilibrium}
The main result of this section is the following

\begin{theorem}\label{conv}
Assume (H1)--(H3). Then, for any initial datum
 $(\theta_0, \mathbf{q}_0, \chi_0, \xi_0, \chi_1)\in \mathbb{Y}$, the unique global weak solution to problem \eqref{c1}--\eqref{c7} satisfies
 \be
 \label{convergence}
 \lim_{t\to +\infty} \|(\theta, \mathbf{q}, \chi, \xi, \chi)(t)-(\theta_\infty, \mathbf{0}, \chi_\infty, \xi_\infty, 0)\|_{\mathbb{X}}=0,
 \ee
 where $(\theta_\infty, \chi_\infty, \xi_\infty)$ solves the stationary problem \eqref{sta}.
 \end{theorem}

 \br
 Recalling Remark \ref{smoothattr}, it can be shown that the solution converges in $\mathbb{Y}-$norm to the single equilibrium.
 \er

 The proof of Theorem \ref{conv} consists of several steps.

\textit{Step 1.} \textbf{Characterization of the $\omega$-limit set}. We define the $\omega$-limit set in $\mathbb{X}$ by
\bea
&&\omega(\theta_0, \mathbf{q}_0, \chi_0, \xi_0, \chi_1)\non\\
&=& \{(\theta^*, \tq^*, \chi_0^*, \xi^*, \chi_1^*): \; \exists \{t_n\}\nearrow +\infty,
\|(\theta, \mathbf{q}, \chi, \xi, \chi_t)(t_n)-(\theta^*, \tq^*, \chi_0^*, \xi^*, \chi_1^*)\|_{\mathbb{X}}\to 0\}.\non
\eea
Then we have
\bp
\label{oset}
Suppose that (H1)--(H3) are satisfied.
For any $(\theta_0, \mathbf{q}_0, \chi_0, \xi_0, \chi_1 )\in \mathbb{X}$, the set $\omega(\theta_0, \mathbf{q}_0, \chi_0, \xi_0, \chi_1 )$ is
non-empty, compact and connected in the strong topology of $\mathbb{X}$.
Moreover,
\be
\omega(\theta_0, \mathbf{q}_0, \chi_0, \xi_0, \chi_1)=\{(\theta_\infty, \mathbf{0}, \chi_\infty, \xi_\infty, 0)\}, \label{ome}
\ee
where $(\theta_\infty, \chi_\infty, \xi_\infty)$ is a solution to \eqref{sta}. Besides, the functional $\Upsilon$ (cf. \eqref{upsilon})
is constant on the $\omega$-limit set.
\ep
\begin{proof}
Due to the uniform estimate \eqref{eesY} for the weak solution  $(\theta(t), \mathbf{q}(t), \chi(t), \xi(t), \chi_t(t))$
in $\mathbb{Y}$ (see Lemma \ref{esY}), the first conclusion follows from the general results in the theory of infinite dimensional
dynamical systems (cf. \cite{Temam}).
Next, we prove \eqref{ome}. Introduce the functional
 \bea \mathcal{E}(t) &=& \frac12(\|\tth(t)\|^2+\|\mathbf{q}(t)\|^2+\|\nabla \tch(t)\|^2+\|\nabla_\Gamma \txi(t)\|_{H_\Gamma}^2+\|\tch_t(t)\|_{V^*}^2)\non\\
 && \ +\int_\Omega F(\tch(t)+\langle \chi_0+\chi_1\rangle) dx+\int_\Gamma G(\txi(t)+\langle \chi_0+\chi_1\rangle)dS\non\\
&& + \kappa_1 \int_\Omega \mathbf{q}(t)\cdot \nabla A_0^{-1}\tth (t) dx\non\\
&=& \Upsilon(\tch (t), \txi (t))+\frac12(\|\tth(t)\|^2+\|\mathbf{q}(t)\|^2+\|\tch_t(t)\|_{V^*}^2)
+ \kappa_1 \int_\Omega \mathbf{q}(t)\cdot \nabla A_0^{-1}\tth (t)dx,\non
\eea
where $\kappa_1$ is a sufficiently small positive constant.
Similarly to the calculations performed in Section 3, we deduce that
\bea
&& \frac{d}{dt}\mathcal{E} + \|\mathbf{q}\|^2+\|\tch_t\|^2_{V^*}+\alpha\|\tch_t\|^2+\|\txi_t\|^2_{H_\Gamma}+\kappa_1\|\tth\|^2\non\\
&=& -Q_1\int_\Gamma \txi_t dS +\int_\Omega (f(\tch+\langle \chi_0+\chi_1\rangle)-f(\chi))\tch_t dx\non\\
&& +\int_\Gamma (g(\txi+\langle \chi_0+\chi_1\rangle)-g(\xi)) \txi_t dS\non\\
&& -\kappa_1\int_\Omega \mathbf{q}\cdot \nabla A_0^{-1} \tth dx -\kappa_1\int_\Omega \mathbf{q}\cdot \nabla A_0^{-1}\tch_tdx
+\kappa_1\|A_0^{-\frac12}\nabla \cdot \mathbf{q}\|^2.\label{dee}
\eea
Recalling \eqref{cv3} and \eqref{cv4}, on account of the growth assumption (H3), the uniform estimate \eqref{disa1} in $\mathbb{X}$ and the Sobolev embedding theorems, then it follows
\bea
-Q_1(t)\int_\Gamma \txi_t (t) dS&\leq& \frac14\|\txi_t(t)\|^2_{H_\Gamma}+ Ce^{-2t},\non\\
\int_\Omega (f(\tch (t)+\langle \chi_0+\chi_1\rangle)-f(\chi)(t))\tch_t (t)dx
&\leq& C|Q_1(t)|\|\tch_t(t)\|\leq \frac{\alpha}{2}\|\tch_t(t)\|^2+Ce^{-2t},\non\\
\int_\Gamma (g(\txi(t)+\langle \chi_0+\chi_1\rangle)-g(\xi(t))) \txi_t(t) dS
&\leq& C|Q_1(t)|\|\txi_t(t)\|_{H_\Gamma}\non\\
&\leq&\frac14\|\txi_t(t)\|^2_{H_\Gamma}+Ce^{-2t}.\non
\eea
The last three terms on the right-hand side of \eqref{dee} can be estimated as in \eqref{kk}. Taking $\kappa_1$ sufficiently small, we can find a
constant $C_0$, depending on the $\mathbb{X}$-norm of the initial datum, $\alpha$, $|\Omega|$ and $|\Gamma|$, such that
\be
\frac{d}{dt}\mathcal{E}(t)+\frac12\|\mathbf{q}(t)\|^2+\frac12\|\tch_t(t)\|^2_{V^*}+\frac{\alpha}{2}\|\tch_t(t)\|^2
+\frac12\|\txi_t(t)\|^2_{H_\Gamma}+\frac{\kappa_1}{2}\|\tth(t)\|^2\leq C_0e^{-2t}, \quad \forall\, t\geq 0.\label{LYA}
\ee
Then, for any $t'\leq t$, we have
\bea
&& \mathcal{E}(t')+\frac12 \int_t^{t'} \|\mathbf{q}(s)\|^2+\|\tch_t(s)\|^2_{V^*}+\alpha\|\tch_t(s)\|^2+\|\txi_t(s)\|^2_{H_\Gamma}+\kappa_1\|\tth(s)\|^2 ds\non\\
&\leq& \mathcal{E}(t)+ C_0\int_t^{t'} e^{-2s} ds.\non
\eea
From Remark \ref{AH}(1) it follows that $\mathcal{E}$ is continuous on $\mathbb{X}$. Due to (H2), $\mathcal{E}$ is bounded from below by a constant.
As a consequence, for some constant $\mathcal{E}_\infty$, it holds
$$\lim_{t\to+\infty}\mathcal{E}(t)=\mathcal{E}_\infty.$$
On the other hand, we  infer from \eqref{LYA} that
\be
\int_0^{+\infty} \Big( \|\tth(t)\|^2+\|\mathbf{q}(t)\|^2+\|\tch_t(t)\|^2_{V^*}+\alpha\|\tch_t(t)\|^2+\|\txi_t(t)\|^2_{H_\Gamma} \Big) dt<+\infty.\label{kkk}
\ee
From the integral control \eqref{kkk}, on account of \eqref{r2}, \eqref{r4}, \eqref{cv3}, we easily deduce
\be
\lim_{t\to+\infty}\|\tth(t)\|=0, \quad \lim_{t\to+\infty}\|\mathbf{q}(t)\|=0, \quad \lim_{t\to+\infty}\|\chi_t(t)\|_{V^*}=0.\label{convt}
\ee
Consequently, any point in $\omega(\theta_0, \mathbf{q}_0, \chi_0, \xi_0, \chi_1)$ is of the form $(\theta_\infty, \mathbf{0}, \chi_\infty, \xi_\infty, 0)$
and we have
\be
\lim_{t\to +\infty}\Upsilon(\tch(t), \txi(t))=\mathcal{E}_\infty.\label{limUP}
\ee
Let $\{t_n\}$ be an unbounded sequence such that $t_{n+1}\geq t_{n}+1$ and
$$ \lim_{t_n\to +\infty} \|(\theta, \mathbf{q}, \chi, \xi, \chi_t)(t_n)-(\theta_\infty, \mathbf{0}, \chi_\infty, \xi_\infty, 0)\|_{\mathbb{X}}=0.$$
We show that $(\theta_\infty, \chi_\infty, \xi_\infty)$ is a solution to the stationary problem \eqref{sta}.
From \eqref{convt}, \eqref{cv1} and \eqref{cv3} it is easy to see that $\theta_\infty=\langle\theta_0\rangle-\langle\chi_1\rangle$.
Next, for any $n$, we denote $(\chi_n(s), \xi_n(s)):=(\chi(t_n+s), \xi(t_n+s))$. When $t_n\to +\infty$, from \eqref{kkk} we deduce
$$
\int_0^1\Big( \|\partial_s\tch_n(s)\|^2_{V^*}+\|\partial_s\txi_n(s)\|^2_{H_\Gamma}\Big) ds <+\infty.
$$
As a result,
$$ \|\tch_n(s_1)-\tch(s_2)\|_{V^*}\to 0, \ \ \|\txi_n(s_1)-\txi_n(s_2)\|_{H_\Gamma}\to 0, \quad \text{uniformly for} \ s_1, s_2\in [0,1].
$$
Combining it with \eqref{cv3} and the precompactness of the trajectory (cf. Proposition \ref{precom}), we infer
\be
(\chi_n(s), \xi_n(s))\to (\chi_\infty, \xi_\infty), \quad \text{strongly in}\ \mathbb{H}^1,\non
\ee
which further yields
\be
\mu_n(s):=\mu(\chi(t_n+s), \theta(t_n+s))\to \mu(\chi_\infty, \theta_\infty), \quad \text{strongly in}\ V^*.\non
\ee
Then, for any $\phi\in D(A_0^\frac12)$, we have
\bea
 && (\mu(\chi_\infty, \theta_\infty), \phi)_{V^*, V}\non\\
 &=& \int_0^1 (\mu(\chi_\infty, \theta_\infty), \phi)_{V^*, V} ds\non\\
 &=& \lim_{n\to +\infty} \int_0^1 (\mu_n(s), \phi)_{V^*, V} ds\non\\
 &=& -\lim_{n\to +\infty} \int_0^1 (A_0^{-1}(\partial_{ss}\chi_{n}(s)+\partial_s\chi_{n}(s)), \phi)_{V^*, V} ds\non\\
 &=& -\lim_{n\to +\infty} \int_0^1 (A_0^{-1}(\partial_s\chi_{n}(s)-\langle\partial_s\chi_{n}(s)\rangle), \phi)_{V^*, V} ds\non\\
 &&  - \lim_{n\to +\infty}  (A_0^{-1}( \chi_t(t_n+1)-\chi_t(t_n)-\langle \chi_t(t_n+1)\rangle+\langle \chi_t(t_n)\rangle), \phi)_{V^*, V}\non\\
 &=& 0,
\eea
which implies that there exists a constant $\tilde{\mu}_\infty$ such that
\be
\mu(\chi_\infty, \theta_\infty)=\tilde{\mu}_\infty.
\ee
Next, for any $(u, v)$ in $\mathbb{H}^1$, we have
\bea
&& (\tilde{\mu}_\infty, u)_{V^*, V}\non\\
&=&(\mu(\chi_\infty, \theta_\infty), u)_{V^*, V}\non\\
&=& \lim_{n\to+\infty}\int_0^1 (\nabla \chi_n(s), \nabla u)+(\nabla_\Gamma \xi_n(s), \nabla_\Gamma v)_{H_\Gamma}
+\alpha(\partial_s \chi_n(s), u)+ (\partial_s\xi_n(s), v)_{H_\Gamma} ds\non\\
&& +\lim_{n\to+\infty} \int_0^1 (f(\chi_n(s)), u)+(g(\xi_n(s)), v)_{H_\Gamma}-(\theta_n(s), u) ds\non\\
&=& (\nabla \chi_\infty, \nabla u)+(\nabla_\Gamma \xi_\infty, \nabla_\Gamma v)_{H_\Gamma}+ (f(\chi_\infty), u)
+(g(\xi_\infty), v)_{H_\Gamma}-(\theta_\infty, u)\non\\
&& +\lim_{t_n\to+\infty} \alpha(\chi(t_n+1)-\chi(t_n), u)+ (\xi(t_n+1)-\xi(t_n), v)_{H_\Gamma}\non\\
&=& (\nabla \chi_\infty, \nabla u)+(\nabla_\Gamma \xi_\infty, \nabla_\Gamma v)_{H_\Gamma}+ (f(\chi_\infty), u)
+(g(\xi_\infty), v)_{H_\Gamma}-(\theta_\infty, u). \label{muin}
 \eea
Thus, we can see that $(\chi_\infty, \xi_\infty, \theta_\infty)$ satisfies the stationary problem \eqref{sta} (in the weak form).
Simply taking $u=v=1$ in \eqref{muin}, we deduce that $\tilde{\mu}_\infty+\theta_\infty=\mu_\infty$ and \eqref{muinf} holds. Finally, \eqref{limUP} implies that the functional $\Upsilon$ is constant on the $\omega$-limit set. The proof is complete.
\end{proof}

\textit{Step 2.} \textbf{Convergence to equilibrium}. In the spirit of  \cite{HJ99,GPS2}, we now consider the functional
\be
\mathcal{G}=( A_0^{-1} \tilde{\chi}_t, A_0^{-1} ( {\rm P_0} ( -\Delta \tilde{\chi} + \widehat{f}(\tilde{\chi}))),\non
\ee
which, by the decay property \eqref{convt} and the uniform estimate \eqref{disa1}, satisfies
$$\lim_{t\to+\infty}\mathcal{G}(t)=0.$$
On the other hand, from \eqref{c3a} we deduce
\be
A_0^{-1}\tch_{tt}+ A_0^{-1}\tch_t +\alpha \tch_t+{\rm P_0}( -\Delta \tch+\widehat{f}(\tch))=\tilde{\theta}+ {\rm P_0}(\widehat{f}(\tch)-f(\chi)).\non
\ee
 Then, using the above relation, we compute
\bea
\frac{d}{dt} \mathcal{G}&=& ( A_0^{-1} \tilde{\chi}_{tt}, A_0^{-1} ( {\rm P_0} ( -\Delta \tilde{\chi}
 + \widehat{f}(\tilde{\chi}))))+( A_0^{-1} \tilde{\chi}_t, A_0^{-1} ( {\rm P_0} ( -\Delta \tilde{\chi}_t + \widehat{f}'(\tilde{\chi})\tch_t)))\non\\
 &=& - ( A_0^{-1} \tilde{\chi}_{t}, A_0^{-1} ( {\rm P_0} ( -\Delta \tilde{\chi}+ \widehat{f}(\tilde{\chi}))))
 -\alpha(\tilde{\chi}_{t}, A_0^{-1} ( {\rm P_0} ( -\Delta \tilde{\chi}+ \widehat{f}(\tilde{\chi}))))\non\\
 && -\| A_0^{-\frac12} ( {\rm P_0} ( -\Delta \tilde{\chi} + \widehat{f}(\tilde{\chi})))\|^2
 + (\tth, A_0^{-1} ( {\rm P_0} ( -\Delta \tilde{\chi} + \widehat{f}(\tilde{\chi})))\non\\
 && +({\rm P_0}(\widehat{f}(\tch)-f(\chi)),  A_0^{-1} ( {\rm P_0} ( -\Delta \tilde{\chi} + \widehat{f}(\tilde{\chi}))))
 +\|A_0^{-\frac12} \tch_t\|^2)\non\\
 && +( A_0^{-1} \tilde{\chi}_t, A_0^{-1} ( {\rm P_0} (\widehat{f}'(\tilde{\chi})\tch_t)).\label{GG}
\eea
From the uniform estimate \eqref{disa1} (cf. Lemma \ref{es}) and the Sobolev embedding theorem we infer
\bea
&& |({\rm P_0}(\widehat{f}(\tch)(t)-f(\chi)(t)),  A_0^{-1} ( {\rm P_0} ( -\Delta \tilde{\chi}(t)+ \widehat{f}(\tilde{\chi}(t)))))|\non\\
&\leq& \frac18  \| A_0^{-\frac12} ( {\rm P_0} ( -\Delta \tilde{\chi}(t) + \widehat{f}(\tilde{\chi}(t))))\|^2
+\|A_0^{-\frac12}{\rm P_0}(\widehat{f}(\tch(t))-\widehat{f}(\tch (t)-\langle \chi_1\rangle e^{-t}))\|^2\non\\
&\leq& \frac18  \| A_0^{-\frac12} ( {\rm P_0} ( -\Delta \tilde{\chi}(t) + \widehat{f}(\tilde{\chi}(t))))\|^2+C e^{-2t},\non
\eea
and
\be
| ( A_0^{-1} \tilde{\chi}_t, A_0^{-1} ( {\rm P_0} (\widehat{f}'(\tilde{\chi})\tch_t)))|\leq C\|A_0^{-\frac12}\tch_t\|^2.\non
\ee
The remaining terms on the right-hand side of \eqref{GG} are easy to handle. Then we have
\bea
\frac{d}{dt} \mathcal{G}(t) \ \leq \ -\frac12 \| A_0^{-\frac12} ( {\rm P_0} ( -\Delta \tilde{\chi}(t) + \widehat{f}(\tilde{\chi}(t))))\|^2
+ C\|A_0^{-\frac12}\tch_t(t)\|^2+C\|\tth(t)\|^2+C e^{-2t}.\label{dGG}
\eea
Let $\kappa_2>0$ be sufficiently small (and possibly depending on $\kappa_1$). We define the functional
\be
\mathcal{H}(t)=\mathcal{E}(t)+\kappa_2\mathcal{G}(t), \quad t\geq 0.\non
\ee
It is easy to see that
$$
\lim_{t\to+\infty}\mathcal{H}(t)=\mathcal{E}_\infty.
$$
We infer from \eqref{LYA} and \eqref{dGG} that
\be
\frac{d}{dt} \mathcal{H}(t)+ \mathcal{D}(t)\leq C_1 e^{-2t},\label{dHH}
\ee
where
\bea
\mathcal{D}(t)&=&\frac12\|\mathbf{q}(t)\|^2+\frac14\|\tch_t(t)\|^2_{V^*}+\frac{\alpha}{2}\|\tch_t(t)\|^2+\frac12\|\txi_t(t)\|^2_{H_\Gamma}
+\frac{\kappa_1}{4}\|\tth(t)\|^2\non\\
&& + \frac{\kappa_2}{2} \| A_0^{-\frac12} ( {\rm P_0} ( -\Delta \tilde{\chi}(t) + \widehat{f}(\tilde{\chi}(t))))\|^2+e^{-2t}.
\eea

For every point $(\theta_\infty, \mathbf{0}, \chi_\infty, \xi_\infty, 0)$ belonging to the $\omega$-limit set,
we set $\chi_\infty^*=\chi_\infty-\langle \chi_\infty\rangle$, $\xi^*_\infty=\xi_\infty-\langle \chi_\infty\rangle$.
We can associate to $(\chi_\infty^*, \xi_\infty^*)$ the numbers $\rho, \beta$ (depending on $(\chi_\infty^*, \xi_\infty^*)$)
given by Lemma \ref{LS}. Then we obtain the covering
$$
\omega(\theta_0, \mathbf{q}_0, \chi_0, \xi_0, \chi_1)\subset \{\theta_\infty\}\times\{\mathbf{0}\}\times \bigcup \mathbf{B}((\chi_\infty, \xi_\infty), \beta) \times \{0\}.
$$
Due to the precompactness of the trajectory in $\mathbb{X}$, we can extract a finite subcovering of the $\omega$-limit set such that
$$
\omega(\theta_0, \mathbf{q}_0, \chi_0, \xi_0, \chi_1)
\subset \{\theta_\infty\}\times\{\mathbf{0}\}\times\bigcup_{i=1}^{m} \mathbf{B}((\chi_\infty^{(i)}, \xi_\infty^{(i)}), \beta^{(i)}) \times \{0\}.
$$
Taking $\rho=\min_{i=1}^m\{\rho^{(i)}\}\in (0, \frac12)$, we infer that the extended \L ojasiewicz--Simon inequality \eqref{ls1} holds with the
uniform constant $\rho$.  From the definition of the $\omega$-limit set, we know that there exists a sufficient large $t_0$ such that
$$(\tilde{\chi}(t), \txi(t))\in \mathcal{U}:=
\bigcup_{i=1}^m \mathbf{B}((\chi_\infty^{(i)}-\langle\chi_0+\chi_1\rangle, \xi_\infty^{(i)}-\langle\chi_0+\chi_1\rangle), \beta^{(i)} ),\quad \forall\, t\geq t_0.$$
 As a result,  from Lemma \ref{LS}  and \eqref{limUP} we deduce, for all $t\geq t_0$,
\be
\|\mathcal{M}(\tilde{\chi}(t),\txi(t))\|_{(\mathbb{H}_0^1)^*}\geq |\Upsilon(\tilde{\chi}(t),\txi(t))-\Upsilon_\infty|^{1-\rho}.\label{lsa}
\ee
Here, we recall that $\Upsilon$ is constant on the $\omega$-limit set and we denote it by $\Upsilon_\infty$.
On the other hand, if $(u, v)\in \mathbb{H}^2$, recalling \eqref{MM1} and \eqref{cwa}, an integration by parts yields
\bea
&& (\mathcal{M}(u,v), (w,w_\Gamma))_{(\mathbb{H}_0^1)^*, \mathbb{H}_0^1 }\non\\
&=& \int_\Omega (\nabla u\cdot \nabla w + \widehat{f}(u)w) dx
+ \int_\Gamma (\nabla_\Gamma v\cdot \nabla_\Gamma w_\Gamma + \widehat{g}(v) w_\Gamma) dS\non\\
&=& \int_\Omega (-\Delta u + \widehat{f}(u))w dx + \int_\Gamma (-\Delta_\Gamma v+ \partial_\nu u +\widehat{g}(v)) w_\Gamma dS,\non
\eea
which easily implies
\be
\|\mathcal{M}(u,v)\|_{(\mathbb{H}_0^1)^*}\leq C(\|\mathrm{P}_0(-\Delta u + \widehat{f}(u))\|_{V^*}
+\|-\Delta_\Gamma v+ \partial_\nu u +\widehat{g}(v)\|_{H_\Gamma}).\label{LSA}
\ee
Recall that we are now dealing with the weak solution such that $(\chi, \xi)\in \mathbb{H}^3\subset \mathbb{H}^2$. Then we have, for $t\geq t_0$,
\bea
&& C(\|\mathrm{P}_0(-\Delta \tch (t)+ \widehat{f}(\tch(t)))\|_{V^*}+\|-\Delta_\Gamma \txi (t)+ \partial_\nu \tch (t)+\widehat{g}(\txi(t))\|_{H_\Gamma})\non\\
&\geq& |\Upsilon(\tilde{\chi}(t),\txi(t))-\Upsilon_\infty|^{1-\rho}.\label{lsaa}
\eea
We now integrate \eqref{dHH} over the interval $[t, +\infty)$, with $t\geq t_0$, obtaining
\be
\int_t^{+\infty}\mathcal{D}(s) ds = \mathcal{H}(t)-\mathcal{E}_\infty +Ce^{-2t}.\label{DDD}
\ee
On the other hand, using the {\L}ojasiewicz-Simon inequality \eqref{LSA}, the uniform estimates
\eqref{eesY} and the fact $\frac{1}{1-\rho}<2$, we deduce that, for all $t\geq t_0$,
\bea
&&|\mathcal{H}(t)-\mathcal{E}_\infty|\non\\
&\leq &  |\Upsilon(\tch(t), \txi(t))-\Upsilon_\infty|+\frac12(\|\tth(t)\|^2+\|\mathbf{q}(t)\|^2+\|\tch_t(t)\|_{V^*}^2) \non\\
 && +\kappa_1 \left|\int_\Omega \mathbf{q}(t)\cdot \nabla A_0^{-1}\tth (t)dx\right|+\kappa_2 |\mathcal{G}(t)|\non\\
 &\leq& C(\|\mathrm{P}_0(-\Delta \tch (t)+ \widehat{f}(\tch(t)))\|_{V^*}+\|-\Delta_\Gamma \txi (t)+ \partial_\nu \tch (t)
 +\widehat{g}(\txi(t))\|_{H_\Gamma})^\frac{1}{1-\rho}\non\\
 && +C(\|\tth(t)\|^2+\|\mathbf{q}(t)\|^2+\|\tch_t(t)\|_{V^*}^2+\|\mathrm{P}_0(-\Delta \tch (t)+ \widehat{f}(\tch(t)))\|_{V^*}^2)\non\\
 &\leq& C\|\tth(t)\|^\frac{1}{1-\rho}+\|\mathbf{q}(t)\|^\frac{1}{1-\rho}+C\|\tch_t(t)\|_{V^*}^\frac{1}{1-\rho}\non\\
 &&+C\|-\Delta_\Gamma \txi (t)+\partial_\nu \tch (t)+\widehat{g}(\txi(t))\|_{H_\Gamma}^\frac{1}{1-\rho}
 + C \|\mathrm{P}_0(-\Delta \tch(t) + \widehat{f}(\tch(t)))\|_{V^*}^\frac{1}{1-\rho}.\non
\eea
Using the dynamic boundary condition, we see that (cf. \eqref{disa1} and \eqref{k2bis})
\bea
&& \|-\Delta_\Gamma \txi(t) + \partial_\nu \tch(t) +\widehat{g}(\txi(t))\|_{H_\Gamma}\non\\
&\leq &\|\txi_t(t)\|_{H_\Gamma}+\|\widehat{g}(\txi(t))-g(\xi(t))\|_{H_\Gamma}+|Q_1(t)|\non\\
&\leq& \|\txi_t(t)\|_{H_\Gamma}+ C|Q_1(t)|+|Q_1(t)|\non\\
&\leq& \|\txi_t(t)\|_{H_\Gamma}+Ce^{-t}.\non
\eea
As a consequence, we find
\bea
|\mathcal{H}(t)-\mathcal{E}_\infty|
&\leq& C\|\tth(t)\|^\frac{1}{1-\rho}+\|\mathbf{q}(t)\|^\frac{1}{1-\rho}+C\|\tch_t(t)\|_{V^*}^\frac{1}{1-\rho}\non\\
 &&+C\|\txi_t(t)\|_{H_\Gamma}^\frac{1}{1-\rho}+C\|\mathrm{P}_0(-\Delta \tch(t) + \widehat{f}(\tch(t)))\|_{V^*}^\frac{1}{1-\rho}+Ce^{-\frac{1}{1-\rho}t}\non\\
 &\leq& C (\mathcal{D}(t))^\frac{1}{2(1-\rho)}.\label{HDd}
\eea
It follows from \eqref{DDD} and \eqref{HDd} that
\be
\int_t^{+\infty} \mathcal{D}(s)ds\leq C (\mathcal{D}(t))^\frac{1}{2(1-\rho)}, \quad \forall\, t\geq t_0.
\ee
Then, applying the abstract result \cite[Lemma 7.1]{FS} (see also \cite[Lemma 4.1]{HT01}), we infer
\be
\int_{t_0}^{+\infty} \sqrt{\mathcal{D}(t)} dt<+\infty,\non
\ee
which implies
\be
\int_{t_0}^{+\infty} (\alpha^\frac12\|\tch_t(t)\|+\|\txi_t(t)\|_{H_\Gamma}) dt<+\infty.\non
\ee
Thus, from the definition of $\tch, \txi$, we have
\be
\int_{t_0}^{+\infty} (\alpha^\frac12\|\chi_t(t)\|+\|\xi_t(t)\|_{H_\Gamma}) dt<+\infty.\non
\ee
This entails the convergence of $(\chi(t), \xi(t))$ in $\mathbb{H}$. Due to the uniform estimate in $\mathbb{Y}$ (cf. \eqref{eesY})
and the compact embedding, we see that there exists a steady state $(\chi_\infty, \xi_\infty)$  such that
\be
\lim_{t\to+\infty} \|(\chi(t), \xi(t)-(\chi_\infty, \xi_\infty)\|_{\mathbb{H}^r}=0, \quad 1\leq r<3.\non
\ee
In summary, we have proved the conclusion of Theorem \ref{conv}.

Using the energy differential inequality \eqref{dHH}, the argument developed in \cite{HJ01} (cf. also \cite{WGZ1, WGZ2})
and the energy method, one can proceed to show the estimate of decay rate. The details are left to the interested readers.
More precisely, the following result can be proven.
\bc
Let the assumption of Theorem \ref{conv} be satisfied. Then we have
\be
\|(\theta, \mathbf{q}, \chi, \xi, \chi)(t)-(\theta_\infty, \mathbf{0}, \chi_\infty, \xi_\infty, 0)\|_{\mathbb{X}}\leq C(1+t)^{-\frac{\rho}{1-2\rho}},\non
\ee
for all $t\geq 0$, where $C$ is a constant depending on the $\mathbb{X}$-norm of the initial datum and on the coefficients of the system,
while $\rho\in (0, \frac12)$ may depend on $(\chi_\infty, \xi_\infty)$.
\ec

\medskip

{\bf Acknowledgments.}  Cecilia Cavaterra was partially supported by by the FP7-IDEAS-ERC-StG Grant \#256872 (EntroPhase). Maurizio Grasselli gratefully acknowledges the support from Shanghai Key Laboratory for Contemporary Mathematics of Fudan University through the Senior Visiting Scholarship.
Hao Wu was partially supported by ``Chen Guang" project supported by Shanghai Municipal
Education Commission and Shanghai Education Development Foundation, National Science Foundation of China 11371098 and Shanghai Center for Mathematical Science.


\end{document}